\documentclass{amsart}

\usepackage{amssymb}
\usepackage{eepic}
\usepackage{color}
\usepackage{mathtools,xparse}
\usepackage[linktocpage=true]{hyperref}
\usepackage[utf8x]{inputenc}

\usepackage{comment}
\usepackage[usenames,dvipsnames,table]{xcolor}
\usepackage{tikz}

\newtheorem{thm}{Theorem}[section]
\newtheorem{cor}[thm]{Corollary}
\newtheorem{lem}[thm]{Lemma}

\newtheorem{prop}[thm]{Proposition}
\newtheorem{rem}[thm]{Remark}
\newtheorem{example}[thm]{Example}
\theoremstyle{definition}

\newtheorem{que}[thm]{Question}
\newtheorem{conj}[thm]{Conjecture}

\numberwithin{equation}{section}

\newcommand{\CHp}{\mathrm{\Lambda}_{p,\mathrm{cb}}^{\mathrm{Schur}}}
\newcommand{\Z}{\mathbf{Z}}

\newcommand{\cH}{\mathcal{H}}
\newcommand{\cM}{\mathcal{M}}
\newcommand{\cN}{\mathcal{N}}
\newcommand{\N}{\mathbf{N}}
\newcommand{\F}{\mathbf{F}}
\newcommand{\R}{\mathbf{R}}

\newcommand{\C}{\mathbf{C}}
\newcommand{\SL}{\mathrm{SL}}
\newcommand{\SO}{\mathrm{SO}}
\newcommand{\PSL}{\mathrm{PSL}}

\newcommand{\schwartz}{\mathcal{S}}

\DeclareMathOperator{\Tr}{Tr}

\DeclareMathOperator{\diag}{diag}

\newcommand\vecteur{\vec{u}}
\newcommand\fonction{f}

\newcommand{\LL}{\mathcal{L}}

\newcommand{\sphere}{\ensuremath{\mathbf{S}}}

\author{Mikael de la Salle}
\title[Kakeya conjecture and $\LL(\SL_n\Z)$]{Kakeya conjecture and High-Rank Lattice von Neumann algebras}
\date{\today}
\begin{document}
\maketitle
\begin{abstract}
  If the non-commutative $L^p$ space of $\SL_{n}(\Z)$ has the completely bounded approximation property for some non-trivial value of $p$, then some form of the Kakeya conjecture holds in dimension $d$, for all $d\leq \frac{n+1}{2}$.

  The proof relies on a spherical analogue of the following question in Euclidean harmonic analysis, that we raise and investigate: does a radially symmetric Fourier multiplier that is bounded on $L_p(\R^d)$ for some $p \neq 2$ necessarily have a continuous symbol? We leave the question open, but we prove that the primitive of such function is smooth in the sense of Zygmund, give some necessary conditions for $L_p$-boundedness in terms of Besov spaces and Littlewood-Paley decomposition for the symbol, and observe that a negative answer implies some form of the Kakeya conjecture in dimension $d$. We then provide spherical forms of these results, which, when combined with a refinement of Lafforgue's rank $0$ reduction, leads to the claimed result.
\end{abstract}

\section{Introduction}

A Banach space has the approximation property AP if the identity operator on it is a limit of finite rank maps for the compact-open topology. A notoriously difficult question has been to find spaces that fail the AP. It was raised by Banach \cite{MR1357166}, advertised by Grothendieck \cite{MR0075539,MR94682}, and solved by Enflo, who constructed the first examples \cite{MR402468}. Since then, many other examples have been found, but only one \emph{natural}: $B(\ell_2)$, the space of all bounded linear operators on a Hilbert space  \cite{MR631090}. A candidate for the first natural separable space without AP is $A=C^*_\lambda(\SL_3(\Z))$. This is still open, but Vincent Lafforgue and the author proved \cite{LafforguedlS} that it fails a natural operator-space variant, the operator approximation property OAP, where one asks for the identity on $A$ to be a limit of finite rank maps $A\to A$ for the compact-open topology on $\mathcal{K}\otimes A$, the tensor product of $A$ with the $C^*$-algebra of compact operators on a Hilbert space.

Stronger, the non-commutative $L_p$ space of the von Neumann algebra $\LL \SL_3(\Z))$ fails the OAP as soon as $\Big|\frac 1 p -  \frac 1 2\Big| > \frac 1 4$ \cite{LafforguedlS}. An equivalent way to phrase this result is the absence of $L_p$-Fourier synthesis with operator coefficients for $\SL_3(\Z)$ in this range of values of $p$: there is no way to reconstruct a general element in $L_p(\LL \SL_3\Z \otimes B(\ell_2))$ from its Fourier coefficients, see \cite[\S 2.1]{MR4680356}. With Tim de Laat we generalized this result to $\SL_{2d-1}(\Z)$, where the condition on $p$ becomes $\Big|\frac 1 p - \frac 1 2\Big| > \frac{1}{2d}$ \cite{MR3781331}. Analogous statements for non-archimedian analogues such as $\SL_{2d-1}(\F_p[T])$ have previously been obtained \cite{LafforguedlS}. As I explain for example in \cite{MR4680356}, there was a serious obstacle to make the argument work for $\Big|\frac 1 p - \frac 1 2\Big| \leq \frac{1}{2d}$, and the following question was left open.
\begin{que}\label{que:OPANCLp} Does the non-commutative $L_p$ space of $\LL \SL_{2d-1}(\Z)$ have the OAP for some $p \neq 2$? For all $\Big|\frac 1 p - \frac 1 2\Big| \leq \frac{1}{2d}$?
\end{que}
The importance of this question is that a positive answer would imply that the von Neumann algebras $\LL \PSL_n \Z$ and $\LL \PSL_m(\Z)$ are not isomorphic as soon as $n$ and $m$ are far apart (resp $|n-m|\geq 2$), as predicted by Connes' famous rigidity conjecture.

It was Javier Parcet who first made the observation that the condition relying $p,d$ in this question are identical to the conditions appearing in the Bochner-Riesz conjecture in Euclidean harmonic analysis. This numerical coincidence motivated some activity trying to attack Question~\ref{que:OPANCLp} with ideas originating from Euclidean harmonic analysis \cite{ParcetRicarddlS,MR4660138,MR4882285}, see also the surveys \cite{MR4680356,parcet2025impactschurmultipliersharmonic}. In this paper we go one step further by proving the first formal implication with the realm of conjectures around the Bochner-Riesz/Kakeya conjecture.

A subset of the Euclidean space $\R^d$ is a \emph{Kakeya set} if it contains a unit segment in every direction. Besicovich showed that for $d\geq 2$, Kakeya sets exist, that are small in the sense that they have Lebesgue measure $0$. But the Kakeya conjecture predicts that they are still quite large in the sense that their dimension is $d$. They are various precise conjectures depending on what sense we mean dimension: Minkowski upper or lower, Hausdorff... They can all be expressed combinatorically in terms of configuration of tubes in $\R^d$. The precise form of the conjecture we will work with is the following. A $\delta$-tube $R$ in $\R^d$ is the image of $[0,1]\times B_{\R^{d-1}}(0,\delta)$ by an affine isometry of $\R^d$; in that case we note $\overline{R}$ the image of $[2,3]\times B_{\R^{d-1}}(0,\delta)$. Define $\fonction_d(\delta)$ as the infimum of the real numbers $\varepsilon>0$ such that there is a finite family of $\delta$-tubes $\{R_j\}$ such that the $\overline{R}_j$ are pairwise disjoint but
  \begin{equation*} \frac{\Big| \bigcup_j R_j \Big|}{\sum_j |R_j|} \leq \varepsilon.
  \end{equation*}
\begin{conj}\label{conj:Kakeya} For every $\varepsilon>0$, $\sup_{\delta > 0} \fonction_d(\delta)\delta^{-\varepsilon}=\infty$.
\end{conj}
The Kakeya conjecture is a classical theorem in dimension $2$ \cite{MR272988}, and a solution (including Conjecture~\ref{conj:Kakeya}) has recently been given in dimension $3$ in a celebrated breakthrough by Wang and Zahl \cite[Theorem 1.12]{wang2025volumeestimatesunionsconvex}; in higher dimension it remains open. There are lots of interesting survey texts on the subject, for example \cite{hickman2025kakeyaconjecturedoescome}. Although we do not prove a formal implication, we think of Conjecture~\ref{conj:Kakeya} as a variant of the conjecture that Kakeya sets have Minkowski upper dimension $d$, which is equivalent to the same condition but only for families of $\delta$-tubes whose direction form a maximal $\delta$-separated subset of $\sphere^{d-1}$, see the discussion in \S~\ref{subsec:besicovich}. As pointed out to me by Hong Wang, Conjecture~\ref{conj:Kakeya} is also closely related to the tube doubling conjecture (\cite[\S 1.6]{wang2025volumeestimatesunionsconvex}).

The next theorem is the the main result of this paper. It explains why Question~\ref{que:OPANCLp} is difficult.

\begin{thm}\label{thm:main} Let $d \geq 2$. If $L^p(\LL \SL_{2d-1}(\Z))$ has the operator space approximation property for some $p \neq 2$, then Conjecture~\ref{conj:Kakeya} holds in dimension $d$.
\end{thm}

More precisely, we prove that $\frac{\fonction_d(\delta)^{|\frac 1 p - \frac 1 2|} \log |\log \delta|}{\delta}$ is not integrable at $0$, see Corollary~\ref{cor:Spbddness_implies_integrable_modulus}. This theorem is proved by combining several independent results that we now describe.

\subsection{Rank $0$ reduction}

The starting point is the idea that I call \emph{rank $0$ reduction} in \cite{MR4680356} and which originates in Vincent Lafforgue's work on strong property (T) \cite{Lafforgue08}. It is an idea that allows one to study complicated groups such as $\SL_N(\R)$ or arithmetic subgroups therein through their compact subgroups. The form that is relevant here was understood in \cite{LafforguedlS} for $\SL_3$ and then later extended \cite{MR3781331} for $\SL_{N\geq 3}$, and for which a concise exposition is given in \cite{MR4680356} (see also \cite{MR3047470,MR3035056} for other simple Lie groups of rank $\geq 2$). It provided the following relationship between the structure of high rank lattice von Neumann algebras and harmonic analysis on the sphere. Define the operator $T_\delta$ on $L_2(\sphere^d)$ by $T_\delta f(x)$ is the average of $f$ on $\{y \in\sphere^d, \langle x,y\rangle=\delta\}$. If $d, p$ are such that $T_\delta$ belongs to the Schatten class $S_p$ (that is $\Tr |T_\delta|^p<\infty$) and the map $\delta\in (-1,1) \mapsto T_\delta \in S_p$ is Hölder-continuous, then rank $0$ reduction allows us to deduce some strong non-approximability of the von Neumann algebra of $\SL_{2d-1}(\Z)$: the associated non-commutative $L_p$ space does not have the OAP. The same holds for other lattices in $\SL_{2d-1}(\R)$, and any lattice in a connected simple Lie group of rank $\geq \max(9,2d-1)$. It was proven in \cite{LafforguedlS,MR3781331} that this condition on the map $\delta\mapsto T_\delta$ holds as soon as $p > \frac{2d}{d-1}$. Moreover, this condition is also sufficient (in case of $d=2$, this was claimed without proof in \cite[end of \S~2.3]{MR4680356}, and proven in much greater generality of arbitrary rank $1$ symmetric spaces of compact type by Guillaume Dumas in \cite[Theorem 3.14]{zbMATH07962858}). I have regarded this fact as an indication that rank $0$ reduction cannot be used to say anything for $2<p \leq \frac{2d}{d-1}$.

The next result illustrates that I was wrong. It improves on \cite{LafforguedlS,MR3781331} in two ways. The first improvement is the straightforward but important observation that the argument in  \cite{LafforguedlS,MR3781331} did not really use the continuity of the map $\delta \mapsto T_\delta$, but rather a consequence which is that some class of Schur multipliers have H\"older-continous symbols. Recall that for a measure space $(X,\mu)$ and a bounded measurable function $m:X\times X\to \C$, the Schur multiplier with symbol $m$ is the operator on $S_2(L_2(X,\mu))$ sending the operator $A$ with kernel $(A_{x,y})_{x,y\in X}$ to the operator with kernel $(m(x,y) A_{x,y})_{x,y \in X}$. If it extends (necessarily uniquely) to a bounded operator on the Schatten class $S_p$, we write $\|m\|_{MS_p(L_2(X,\mu))}$ for its norm. The second improvement is that the needed regularity is much weaker than H\"older-continuity: for example a modulus of continuity of order $|\log \delta|^{-\alpha}$ or even $(\log |\log \delta|)^{-1-\alpha}$ with $\alpha>0$ is enough.

\begin{thm}\label{thm:rank0reductionImproved} Let $1 \leq p \leq \infty$ and $d \geq 2$. Assume that there is a non-decreasing function $\varepsilon:[0,1]\to [0,\infty)$ such that:
  \begin{itemize}
  \item for every continuous function $\varphi:[-1,1]\to \R$ and every $\delta \in [-1,1]$,
  \[ |\varphi(0) - \varphi(\delta)| \leq \varepsilon(|\delta|) \|\varphi(\langle \cdot,\cdot\rangle)\|_{MS_p(L_2(\sphere^{d}))},\]
  \item $\delta \mapsto \frac{\varepsilon(\delta)}{\delta |\log \delta|}$ is integrable at $0$.
  \end{itemize}

  Then $L_p(\LL\Gamma)$ does not have the operator space approximation property for $\Gamma= \SL_{2d-1}(\Z)$, a lattice in $\SL_{2d-1}(\R)$, or (if $d=2$ or $d\geq 5$) in a connected simple Lie group of rank $\geq 2d-1$.
  \end{thm}

Theorem~\ref{thm:rank0reductionImproved} opens the possibility to prove the lack of OAP of $L_p(\LL \Gamma)$ for new values of $p$ by proving continuity of Schur multiplier symbols by other means than by considering the spectral decomposition of the operators $T_\delta$. This is exactly what we will do in the rest of the paper.

\subsection{Continuity for the symbol of radial Schur and Fourier multipliers}
Theorem~\ref{thm:rank0reductionImproved} motivates the following question, to which we devote much of the paper.
\begin{que}\label{question:mainSchur} Is there a discontinuous function $m:(-1,1)\to \R$ such that the Schur multiplier with symbol $(\xi,\eta) \in \sphere^d \times \sphere^d \mapsto m(\langle \xi,\eta\rangle)$ is $S_p$-bounded for some $p \neq 2$?
\end{que}

It is now understood that there are strong analogies between Schur multipliers and Fourier multipliers, but Schur multipliers are often more difficult to study. A recent example is \cite{MR4882285}, where, using Schur multipliers, we extended Fefferman's ball multiplier theorem \cite{zbMATH03370942} to arbitrary Lie groups and characterized the idempotent $L_p$-bounded Fourier multipliers in terms of one-dimensional dynamics, see also the survey \cite{parcet2025impactschurmultipliersharmonic}. A Fourier multiplier is a linear map sending a function $f$ on $\R^d$ to the "function" whose Fourier transform is the pointwise multiplication $\varphi \hat f$, for some function $\varphi \colon \R^d\to \C$ called the symbol of the Fourier multiplier. It is said to $L_p$-bounded if it maps $L_p(\R^d)$ to itself, and radial if $\varphi(\xi)$ only depends on the Euclidean norm $|xi|$. Here and throughout the paper we use the convention $\hat f(\xi) = \int f(x) e^{-2i\pi \langle x,\xi\rangle} dx$ for the Fourier transform.

This motivates, as a baby case of Question~\ref{question:mainSchur}, the following question. It is very natural from the non-commutative perspective, but it does not seem to have been extensively studied.
\begin{que}\label{question:mainfourier} Is there a discontinuous function $m:(0,\infty)\to \R$ such that the Fourier multiplier with symbol $\xi \in \R^d  \mapsto m(|\xi|)$ is bounded on $L_p(\R^d)$ for some $p \neq 2$?
\end{que}

The study of radial multipliers that are $L_p$-bounded is a classical subject that has been much investigated, and it is well understood that there are two very different regimes:

\begin{itemize}
  \item The regime when $p$ is far from $2$ (that is $\big| \frac 1 p - \frac 1 2\big| > \frac 1 {2d}$). This regime has been very much investigated. Many questions remain but, at least conjecturally, the picture is rather clear (and important cases of the conjecture have been proved \cite{MR2784663}). For example, in the simple case when the support of $m$ is a compact subset of $(0,\infty)$ : if $p<2$ the $L_p$ boundedness of the radial Fourier multiplier with symbol $m(|\cdot|)$ should be equivalent to the (apparently much weaker condition) that $m(|\cdot|)$ is the Fourier transform of an $L_p$ function. In particular, the $L_p$-boundedness of the radial Fourier multiplier with symbol $m(|\cdot|)$ forces some Hölder continuity of $m$ outside of $0$, and Question~\ref{question:mainfourier} has a negative answer in this regime.

 \item The regime when $p$ is close to $2$ (that is $\big| \frac 1 p - \frac 1 2\big| \leq \frac 1 {2d}$). Here the main result that I am aware of is Fefferman's celebrated Ball multiplier theorem \cite{zbMATH03370942}, which in our setting says that the indicator function of the euclidean ball is not an $L_p$-bounded Fourier multiplier for any $p\neq 2$.
\end{itemize}

The main results of this paper address Questions~\ref{question:mainfourier} and~\ref{question:mainSchur}. While they are stated for arbitrary values of $p$ and $d$, they are mostly relevant in the regime where $p$ is close to $2$.

The first main result is a qualitative answer to Question~\ref{question:mainfourier}. Recall that a function $F:\R\to\C$ is smooth in Zygmund's sense  \cite{zbMATH03096429} if
\[ \lim_{h\to 0} \sup_{t \in \R} \frac{|F(t+h)+F(t-h)-2F(t)|}{h} = 0.\]
\begin{thm}\label{thm:mainqualitative} Let $m:(0,\infty) \to \R$ a bounded measurable function. Assume that, for some $d \geq 2$ and some $p\neq 2$, the Fourier multiplier with symbol $m(|\cdot|)$ is bounded on $L_p(\R^d)$. Then the function $F(t) = \int_0^t m(\exp(s)) ds$ is a smooth function in Zygmund's sense.% Equivalently $\varphi:t\in \R \mapsto m(\exp t)$ belongs to $b_{\infty,\infty}^0(\R)$.
\end{thm}
The prototypical example of a bounded measurable function whose primitive is not smooth is the indicator function of $(0,1)$, so this theorem generalizes Fefferman's ball multiplier theorem. The proof is an adaptation of its proof. It also uses a functional analytic characterization of functions satisfying the conclusion of the theorem (Proposition~\ref{prop:characterization_of_binftyinfty0}). It is worth pointing out that this theorem does not answer Question~\ref{question:mainfourier} because there are discontinuous bounded measurable functions with smooth primitive, see Example~\ref{ex:sinloglog}. This example indicates that candidates for Question~\ref{question:mainfourier} have to be looked among highly oscillating functions. The next question points out functions that lie at the edge of current techniques. This question is already open in dimension $2$.
\begin{que} Can a function equal to $\xi \mapsto \sin(\log \log 1/|1 - |\xi|^2| |)$ on a neighbourhood of the unit sphere define a bounded $L_p$-multiplier for some $p\neq 2$?
\end{que}

The second main result is a quantitative answer, that makes the preceding more precise. Observe indeed that bounded measurable functions with smooth primitive coincide with $L_\infty(\R)\cap b_{\infty,\infty}^0(\R)$, see Proposition~\ref{prop:characterization_of_binftyinfty0}. Recall that, if $(W_n)_{n \geq 0}$ is a Littlewood-Paley partition of unity (see Section~\ref{sec:LP}), the inhomogeneous Besov spaces on $\R$, $B_{p,q}^s(\R)$ and $b_{p,\infty}^s(\R)$ are the spaces of tempered distributions
\[ B_{p,q}^s(\R) = \big\{ f \mid (2^{ns}\|W_n\ast f\|_p)_{n\geq 0} \in \ell_q\big\},\]
and
\[ b_{p,\infty}^s(\R) = \big\{ f \mid (2^{ns}\|W_n\ast f\|_p)_{n\geq 0} \in c_0\big\}.\]

\begin{thm}\label{thm:Lpbddness_LP_quantitative} Let $1<p<\infty$ and $m:(0,\infty) \to \R$ a bounded measurable function such that $T$, the Fourier multiplier with symbol $m(|\cdot|)$ is bounded on $L_p(\R^d)$. Then $\varphi:t\in \R \mapsto m(\exp t)$ satisfies
  \[ \|W_n\ast \varphi\|_\infty \leq C_d \|T\|_{L_p \to L_p} \inf_{\delta \in (0,1)} \left(\fonction_d(\delta)^{\big|\frac 1 p - \frac 1 2\big|} + \frac{1}{2^n \delta^2} \right)\]
for all $n \geq 0$. In particular, if $d=2$ and $n \geq 1$,
  \[ \|W_n\ast \varphi\|_\infty \leq C'_2 \|T\|_{L_p \to L_p} n^{-\big|\frac 1 p - \frac 1 2\big|} \]
  and if $F(t)=\int_0^t \varphi(s) ds$ and $h>0$,
  \[ \sup_{t \in \R} \frac{|F(t+h)+F(t-h)-2F(t)|}{h} \leq C'_2 \|T\|_{L_p \to L_p} |\log h|^{-|\frac 1 p - \frac 1 2|}.\]
\end{thm}
As a consequence, we get that if Conjecture~\ref{conj:Kakeya} fails for $\varepsilon$, then this forces $\varphi$ to belong to $B^{\alpha}_{\infty,\infty}(\R)$ for some $\alpha=\frac{\varepsilon|p-2|}{4p+\varepsilon|p-2|}>0$, or equivalently if $\alpha<1$, $\varphi$ to be $\alpha$-H\"older continuous (see Lemma~\ref{lem:infi_for_particular_fd}). More generally,
\begin{cor}\label{cor:Lpbddness_LP_quantitative} Assume that there is a $L_p$-bounded radial multiplier on $\R^d$ with discontinuous symbol. Then $\int_0^1 \frac{\fonction_d(\delta)^{|\frac 1 p - \frac 1 2|}}{\delta}=\infty$.
\end{cor}
When  $\int_0^1 \frac{\fonction_d(\delta)^{|\frac 1 p - \frac 1 2|}}{\delta}<\infty$, we derive explicit formulas for the modulus of continuity (see Corollary~\ref{cor:Lpbddness_LP_quantitative_precise}).

Finally, developping the ideas that we understood with Parcet and Tablate \cite{MR4882285} and previous transference techniques \cite{MR2866074,CaspersdlS}, we can adapt the proofs to obtain the following spherical analogue of Theorem~\ref{thm:Lpbddness_LP_quantitative}. %
\begin{thm}\label{thm:Spbddness_LP_quantitative} There is a constant $C_d$ such that for every bounded measurable function $m:(-1,1)\to \C$,
  \[ |W_n \ast (m\circ \cos)(\theta)| \leq C_d \|m(\langle \cdot,\cdot\rangle)\|_{MS_p(L_2(\sphere^d))} \left( \inf_{\delta>0} \fonction_d(\delta)^{\big|\frac 1 p - \frac 1 2\big|} + \frac{1}{|\delta^2 2^n \sin \theta|^{\frac 1 3}}\right)\]
for all $n \geq 0$. In particular if $d\geq 2$ and $n\geq 1$,
  \[ |W_n \ast (m\circ \cos)(\theta)| \leq C'_2 \|m(\langle \cdot,\cdot\rangle)\|_{MS_p(L_2(\sphere^2))} \max(1,n+\log|\sin \theta|)^{\big|\frac 1 p - \frac 1 2\big|}.\]
\end{thm}
As for Fourier multipliers, this theorem implies that such Schur multipliers have continuous symbols with explicit modulus if $\fonction_d(\delta)$ decays fast enough. Together with Theorem~\ref{thm:rank0reductionImproved}, this implies Theorem~\ref{thm:main}. See Corollary~\ref{cor:Spbddness_LP_quantitative_precise} and \ref{cor:Spbddness_implies_integrable_modulus}, .

\subsection*{Organization} The first section after this introduction is devoted to various preliminaries on Schur multipliers, non-commutative $L_p$ spaces, the Besicovich construction, Besov spaces and partitions of unity. It consists mostly of reminders, although Propositions~\ref{prop:characterization_of_binftyinfty0} and \ref{prop:from_one_h_to_all} might not already appear in print. Section~\ref{sec:AP_Schur} is devoted to the proof of Theorem~\ref{thm:rank0reductionImproved}. The next two sections are devoted to the study of radial Fourier multipliers in the Euclidean space. They are fully independent from operator algebras and non-commutative analysis. First in section~\ref{sec:qualitative} we prove Theorem~\ref{thm:mainqualitative}. The proof is similar to that in Fefferman's Ball Multiplier Theorem and can be efficiently divided in several independent steps. The next section~\ref{sec:quantitative} proves Theorem~\ref{thm:Lpbddness_LP_quantitative} and its corollaries. The proof exploits several of the previous ingredients, but the steps have to be all performed at once. Finally in the last section we prove Theorem~\ref{thm:Spbddness_LP_quantitative} and its corollaries.
\subsection*{Acknowledgements} I thank Javier Parcet and Eduardo Tablate for numerous and useful interactions and encouragements, Charles Fefferman for inspiring discussions on Question~\ref{question:mainfourier} during my stay at the IAS in the year 2023-2024, and Jean-Yves Chemin, Petru Mironescu and Liding Yao for enlightening exchanges about function spaces, in person or online. I am also grateful to Mark Lewko and Hong Wang for useful comments about the Kakeya conjectures. The first spark that led to this work was initiated during a stay at the ICMAT in the Severo Ochoa Laboraties. Part of this work was also performed during the author's visit to the Isaac Newton Institute during the 2025 programme \emph{Operators, Graphs, Groups}, partially supported by a grant from the Simons Foundation. The author's research was also partially supported by the ANR project ANR-24-CE40-3137.
\section{Preliminaries}

%% \[ \inf_{0<\delta} \delta^{\alpha} + \frac{1}{A \delta^\beta} \leq 2 A^{-\frac{\alpha}{\alpha+\beta}}.\]

%% \[ \inf_{0<\delta<1} |\log \delta|^{-\alpha} + \frac{A}{\delta^{\beta}} \leq C(\alpha,\beta) (\log A)^{-\alpha}. \]

\subsection{Schur multipliers}
If $\cH$ is a complex Hilbert space, and $1\leq p<\infty$, we denote by
$S_p(\cH)$ (or simply $S_p$ is $\cH$ is understood) the Schatten
$p$-class, which is the space of bounded linear operators $T:\cH \to \cH$ such
that $\Tr((T^*T)^{p/2})<\infty$. It is a Banach space for the norm
$\|T\|_p = \Tr((T^*T)^{p/2})^{\frac 1 p}$, and $S_2 \cap S_p$ is
always dense in $S_p$. If $p=\infty$ $S_\infty(\cH)$ is the space of
compact operators on $\cH$ with operator norm.

Consider now the special case when $\cH=L_2(X,\mu)$ for a measure space $(X,\mu)$. Then the space $S_2(L_2(X,\mu))$ coincides isometrically with the space $L_2(X\times X,\mu \otimes \mu)$, where we identify $a \in L_2(X\times X,\mu\otimes \mu)$ with the operator $A$ defined by $Af(x) = \int a(x,y) f(y) d\mu(y)$. Therefore if $m \in L_\infty(X\times X,\mu \otimes \mu)$, then the linear map sending $A = (a(x,y))$ to $S_m(A):= (m(x,y) a(x,y))$ is a bounded operator on $S_2(L_2(X,\mu))$ that is called the Schur multiplier with symbol $m$. If $S_m$ maps $S_2 \cap S_p$ into $S_2 \cap S_p$ and this map extends (necessarily uniquely) continuously to $S_p$, we say that $S_m$ is $S_p$-bounded and we still write $S_m$ for the extension and $\|m\|_{MS_p}$ for its norm.

We shall need the following basic useful fact about restrictions.
\begin{lem}\label{lem:restriction} Let $X,Y$ be locally compact spaces with $\sigma$-finite Radon measures $\mu,\nu$, and $m:X\times X\to \C$ and $f,g:Y\to X$ be Borel measurable functions. Assume either that (1) $m,f,g$ are continuous and $\mu$ has full support, or that (2) the images of $\nu$ by $f$ and $g$ are absolutely continuous with respect to $\mu$.
  \[ \|m\circ (f\times g)\|_{MS_p(L_2(Y,\nu))} \leq \|m\|_{MS_p(L_2(X,\mu))}.\]
\end{lem}
\begin{proof} Under the assumption (1), the lemma is an immediate consequence of \cite[Theorem 1.19]{LafforguedlS}. Under the assumption (2), this is \cite[Lemma 2.1]{MR4882285}.
\end{proof}

\subsection{Non-commutative $L_p$ spaces and completely bounded maps}\label{subsection:NCLp}
The definition of Schatten classes recalled above is a particular case of the more general construction of $L_p$-spaces of a von Neumann algebra, and so is the classical notion of $L_p$ space over a measure space. Such general non-commutative $L_p$ spaces or maps between them will not appear often (only in \S~\ref{sec:AP_Schur} and in the proof of Lemma~\ref{lem:square_inequality_Schur}), so we only briefly recall some definitions for von Neumann algebras that admit normal semifinite faithful (nsf) traces, which are the only algebras encountered in this paper. We refer for example to \cite{PisierOS,PisierXu} for general references. Let $(\cM,\tau)$ be a von Neumann algebra with a nsf trace $\tau$. The set $\{x \in \cM \mid \tau((x^*x)^p)<\infty\}$ is a weak-* dense ideal in $\cM$, and the quantity $\|x\|_p:= \tau((x^*x)^p)^{\frac 1 p}$ is a norm on it. The completion is the non-commutative $L_p$ space of $\cM$, and is denoted $L_p(\cM,\tau)$, or simply $L_p(\tau)$. When $\cM=B(\cH)$ with the usual trace $\Tr$, we recover $S_p(\cH)$. When $\cM=L_\infty(X,\mu)$ for a semifinite measure space $(X,\mu)$ with trace $\tau(f)=\int f d\mu$, we recover $L_p(X,\mu)$.

Pisier discovered \cite{MR1648908} that $L_p(\cM,\tau)$ admits a natural operator space structure, which allows one to talk about completely bounded maps between $L_p(\cM,\tau)$ and another operator space. He also proved that for linear maps $u:L_p(\cM,\tau_N) \to L_p(\cN,\tau_\cN)$ between non-commutative $L_p$ spaces with the same $p$, it has a much simpler characterization:
\[ \|u\|_{cb} = \sup_n \|\mathrm{id}_{M_n} \otimes u \|_{L_p(M_n(\cM),\Tr\otimes \tau_\cM) \to L_p(M_n(\cN),\Tr\otimes \tau_\cN)},\]
where $\Tr\otimes \tau_\cM(x)=\sum_{i=1}^n \tau_{\cM}(x_{i,i})$ is the natural trace on $M_n(\cM)$.

\subsection{The Besicovich construction and comparison with usual forms of the Kakeya conjecture}\label{subsec:besicovich}
Let $0<\delta<1$. A $\delta$-tube $R$ in $\R^d$ is the image of $[0,1]\times B(0,\delta) = \{(x,y) \in \R\times \R^{d-1} \mid 0\leq x \leq 1, |y|<\delta\}$ by an affine isometry of the Euclidean space $\R^d$, together with an orientation that is given by the vector $\vecteur(R)$ pointing from the image of $(0,0)$ to the image of $(1,0)$. Given a $\delta$-tube $R$, we write
\[\overline{R} = R+2\vecteur(R).\]
The precise choice of $2$ in the definition $\overline{R} = R+2\vecteur(R)$ is not important. All the results of the paper remain true if it is replaced by another number $>1$ and more generally if $\overline R:=\overline{R}^{[a,b]}$ is the image of $[a,b] \times B(0,\delta)$ for some $a<b$, only the constants differ. 

For every $\delta \in (0,1)$, in the introduction we defined $\fonction_d(\delta)$ as the infimum of the real numbers $\varepsilon>0$ such that there is a finite family of $\delta$-tubes $\{R_j\}$ such that the $\overline{R}_j$ are pairwise disjoint but
  \begin{equation}\label{eq:gooddeltatubes} \frac{\Big| \bigcup_j R_j \Big|}{\sum_j |R_j|} \leq \varepsilon.
  \end{equation}
  More generally we write $\fonction_d^{[a,b]}(\delta)$ if we require that the $\overline{R}_j^{[a,b]}$ are pairwise disjoint.
  
  Keich's construction \cite{zbMATH01344340} of Besicovitch sets gives that
  \begin{equation}\label{eq:Keich}\frac{c}{|\log \delta|} \leq \fonction_2(\delta) \leq \frac{C}{|\log \delta|}\end{equation}
  and the same holds, with different $c,C$ for $\fonction_d^{[a,b]}$ when $1<a<b\leq \infty$. In higher dimension not much in known as far as upper bounds are concerned, except the trivial bound $\fonction_d \lesssim \fonction_2$.

  It is worth pointing out that the requirement that the $\delta$-tubes $\overline{R_j}$ are pairwise disjoint is not exactly what is needed for the constructing Kakeya sets, but it is a feature of the constructions of Kakeya sets (at least those known to me). The natural requirement is that the directions of the $\delta$-tubes forms a maximal $\delta$-net, so that ${\sum_j |R_j|}\simeq 1$. Morally, the two conditions are similar: if two $\delta$-tubes  have directions that make an angle $\geq C \delta$ and they intersect (which should happen for many different pairs if we want \eqref{eq:gooddeltatubes} to be small), then their translates are pairwise disjoint. But the conditions do not coincide, so we can define a variant $\widetilde{\fonction}_d(\delta)$ as the smallest real number $\varepsilon>0$ such that there is a family of $\delta$-tubes satisfying \eqref{eq:gooddeltatubes} whose directions form a maximal $\delta$-separated subset of $\sphere^{d-1}$. Then it is easy to see that Conjecture~\ref{conj:Kakeya} for $\widetilde{\fonction}_d$ is equivalent to the conjecture that all Kakeya sets have upper Minkowski dimension $d$. More generally, all Kakeya sets have upper (lower) dimension $\geq d-c$ if and only if
  \[ \limsup \widetilde \fonction_d(\delta) \delta^{-c'} = \infty\ \  (\mathrm{respectively }  \liminf \widetilde \fonction_d(\delta) \delta^{-c'}=\infty)\]
  for every $c'>c$.

As a naive observer of the question, I have hard time imagining how it could be that $\widetilde{\fonction}_d$ is much smaller than $\fonction_d$, so I would expect that the asymptotic behaviour of $\fonction_d(\delta)$ as $\delta\to 0$ indeed captures the dimension of all Kakeya sets. In particular, I would expect that Conjecture~\ref{conj:Kakeya} implies that all Kakeya sets have Minkowski upper dimension $d$, and Conjecture~\ref{conj:Kakeya} with $\liminf$ instead of $\limsup$ implies that all Kakeya sets have Minkowski lower dimension $d$.

I give some straightforward evidence in this sense. %The argument is a bit informal, and to make it precise we need to allow other values than $2$ in the definition $\overline{R}=R+2\vecteur(R)$.
Let $R,R'$ be two $\delta$-tubes with $\delta$-separated directions, $L,L'$ the supporting affine lines and $\vecteur,\vecteur' \in \sphere^{d-1}$ the supporting directions. Then for their translates $\overline{R}^{[a,b]}$ and $\overline{R'}^{[a,b]}$ to intersect, we need two things: (1) $L$ and $L'$ must be almost coplanar in the sense that $d(L,L') \leq 2\delta$, and (2) there is a point in $L^\delta \cap (L')^\delta$ at distance $\geq a-1$ from $R$ and $R'$. This does not happen if $d(R,R') \leq C d(\vecteur,\vecteur')$ and $a$ is large enough. As a consequence, we obtain
\begin{lem} Every Kakeya set of the form $\bigcup_{u \in \sphere^{d}} p(u)+[0,1]u$ for a Lipschitz function $p:\sphere^{d}\to \R^d$ have Minkowski upper dimension $\geq d_+$ and Minkowski lower dimension $\geq d-$, where
  \[ d_+ = d-\inf_{a>1}\liminf_{\delta\to 0} \frac{\log \fonction_d^{[a,\infty]}(\delta)}{\log \delta},\]
\[  d_- = d-\inf_{a>1} \limsup_{\delta\to 0} \frac{\log\fonction_d^{[a,\infty]}(\delta)}{\log \delta}.\]
\end{lem}
In particular, Conjecture~\ref{conj:Kakeya} for $\fonction_d^{[a,b]}$ and for all $1<a<b$ implies that every such Kakeya set has Minkowski upper dimension $d$. Perhaps the same holds for sticky Kakeya sets as defined in \cite{Wang_2025}, and maybe even for all Kakeya sets, but we leave that question to experts.

 %% To state it, we need to introduce, for every $d \geq 2$, some function $\fonction_d\colon (0,1) \to (0,1)$ that measures the strength of Besicovich sets in $\R^d$ (see the beginning of Section~\ref{sec:quantitative} CHANGE!). Let us briefly recall that Besicovich's theorem is that $\lim_{\delta \to 0} \fonction_d(\delta)=0$, and more precisely Keich's construction of Besicovich sets asserts that
%% \[ \fonction_d(\delta) \leq \frac{C}{\log(2/\delta)},\]
%% whereas Conjecture~\ref{conj:Kakeya} is equivalent to that
%% \[ \forall \varepsilon>0, \sup_{\delta \in (0,1)} \delta^{-\varepsilon} \fonction_d(\delta)>0.\]
%% \begin{proof}[Proof of the implication] Assume that 
%% \[ \forall \varepsilon>0, \sup_{\delta \in (0,1)} \delta^{-\varepsilon} \fonction_d(\delta)>0.\]
%% We prove that every Kakeya set in $\R^d$ has upper Minkowski dimension $d$. Let $X \subset \R^d$ be a Kakeya set, and $X(\delta)$ its $\delta$-neighbourhood. We have to prove that $\limsup \frac{\vol(X(\delta))}{- \log \delta} \geq 0$. Let $\varepsilon>0$. Without loss of generality we can assume that $X$ is a union of unit segments whose center belongs to $[-1/2,1/2]^d$ (indeed, if for $k \in \Z^d$ we denote by $C(k,X)$ the union of all unit intervals included in $X$ and whose center belongs to $k+[-1/2,1/2]^d$, then $Y:=\bigcup_{x \in \Z^d} C(k,X)-k$ is a Kakeya set with the required properties, and $\vol(Y(\delta)) \leq 2^d \vol(X(\delta))$).
%% \end{proof}

\subsection{Besov spaces and Littlewood-Paley decomposition}\label{sec:LP}

Write $\schwartz(\R)$ for the Schwartz space on $\R$.
  
Let $w_0,w \in C^\infty_c(\R)$ be nonnegative $C^\infty$ functions with compact support contained in $[-2,2]$ and $[\frac 1 2,2]$ respectively and such that, if for $n\geq 1$ we set $w_n(x) = w(|x|/2^n)$ we have
\begin{equation}\label{eq:Littlewood-Paley} \forall x \in \R, w_0(x) + \sum_{n\geq 1} w_n(x)=1.
\end{equation}
  Let $W_n$ be the Schwartz function whose Fourier transform is $w_n$. The Besov space $B_{p,q}^s(\R)$ is defined as the space of tempered distributions such that $W_n \ast f \in L_p(\R)$ for all $n$ and
  \[ (2^{ns}\|W_n\ast f\|_p)_{n\geq 0} \in \ell_q.\]
  When $q=\infty$, we define $b_{p,\infty}^s(\R)$ as the subspace of $B_{p,\infty}^s(\R)$ satisfying
\[ (2^{ns}\|W_n\ast f\|_p)_{n\geq 0} \in c_0.\]
It is very elementary that bounded continuous functions are in $b_{\infty,\infty}^0(\R) \cap L_\infty(\R)$. The example below shows that the inclusion is strict. It is due to Brezis and Nirenberg, it was kindly communicated to me by Liding Yao on Mathoverflow.
\begin{example}\cite[Example 4]{MR1354598}\label{ex:sinloglog} The function $f(t) = \sin(\log |\log |t||) \varphi(t)$, where $\varphi$ is a $C^\infty$ function with support contained in $(-1,1)$ and $\varphi(0)>0$ is bounded measurable, belongs to $b_{\infty,\infty}^0(\R)$ but is not continuous.
\end{example}
In the next statement, $\sigma(L_\infty,L_1)$ is the weak-* topology on $L_\infty(\R)$.
\begin{prop}\label{prop:characterization_of_binftyinfty0} Let $f \in L_\infty(\R)$. The following are equivalent:
  \begin{enumerate}
  \item\label{item:besov0} $f \in b_{\infty,\infty}^0(\R)$,
  \item\label{item:w*cv} all $\sigma(L_\infty,L_1)$-limits of sequences of the form $f(s_n + \frac{\cdot}{r_n})$ with $s_n,r_n \in \R$ and $r_n \to \infty$ are constants.
  \item\label{item:Zygmund} The function $F(t)=\int_0^t f(s) ds$ is smooth in Zygmund's sense:
    \begin{equation}\label{eq:zygmund} \lim_{h \to 0}\sup_x \frac{|F(x+h)+F(x-h)-2F(x)|}{h}=0.
  \end{equation}
  \end{enumerate}
\end{prop}
The proof relies on the following lemma:
\begin{lem}\label{lem:density} The linear space spanned by the functions
  \[ \{ \frac{1}{h}(\chi_{[x,x+h]} - \chi_{[x,x-h]}) \mid x\in \R, h \in (0,1]\}\]
  is dense in $L_1^0(\R)$.
\end{lem}
\begin{proof}  By Hahn-Banach, we have to prove that the only functions $f \in L_\infty(\R)$ (the dual of $L_1(\R)$) such that
  \[ \int \frac{1}{h}(\chi_{[x,x+h]} - \chi_{[x-h,x]}) f=0\]
  for every $x \in \R$, $h \in (0,1]$ are the almost everywhere constant functions. Fix $x_0\in \R$ and for $n \in \N$ let $\mathcal{A}_n = \sigma([x_0+k2^{-n}, x_0+(k+1) 2^{-n}[, k \in \Z)$. Let $f_n$ be the conditional expectation of $f$ with respect to $\mathcal{A}_n$, so that the preceding says that $f_n=f_{n+1}$ for every $n\geq 0$. As a consequence, we get that $f_0 = \lim_n f_n = f$ almost everywhere by the martingale convergence theorem. Applying this for $x_0=0$ and $x_0=\frac 1 2$, we see that this forces $f$ to be constant both on intervals $[k,k+1]$ and $[k-1/2,k+1/2]$ for $k \in \Z$, which forces $f$ to be constant.
\end{proof}
We also need the following standard fact, essentially Bernstein's inequality:
\begin{lem}\label{lem:Bernstein}
  There is a constant $C$ such that for any $r>0$ and any function $g \in L_1(\R)$ whose Fourier transform is supported in $[-2r,-r/2]\cup[r/2,2r]$,
  \[ \frac{r}{C} \|g\|_\infty \leq \|g'\|_\infty \leq C r \|g\|_\infty.\]
\end{lem}
\begin{proof} Considering the function $x\mapsto g(x/r)$, we can assume that $r=1$. Let $\phi,\psi \in \mathcal{S}(\R)$ such that $\hat \phi(t) = -2i\pi t$ and $\hat \psi(t) = \frac{-1}{2i\pi t}$ on $[-2,-1/2] \cup [1/2,2]$. The functions $g'$ and $\phi \ast g$ have the same Fourier transform so they are equal, and similarly $g =\psi \ast g'$, so the Lemma holds with $C=\max(\|\psi\|_1,\|\psi\|_1)$.
\end{proof}
\begin{proof}[Proof of Proposition~\ref{prop:characterization_of_binftyinfty0}] Let $f \in L_\infty(\R)$ and $F$ be the primitive of $f$ vanishing at $0$, and $F_n=W_n \ast F$. It follows from Lemma~\ref{lem:Bernstein} with $g= F_n$ and $n\geq 1$ that $f \in b_{\infty,\infty}^0(\R)$ if and only if $\lim_n 2^{n} \|W_n \ast F\|_\infty =0$, which is well-known to be equivalent to \eqref{eq:zygmund}. We recall the argument for completeness. The implication $\lim_n 2^{n} \|W_n \ast F\|_\infty =0 \implies$\eqref{eq:zygmund} is for example \cite[Corollary 2.7]{MR2628799}. For the converse write $F = \sum_{n\geq 0}F_n$. Using Taylor expansion and Bernstein's inequality as in Lemma~\ref{lem:Bernstein}, bound
  \[ |F_n(x+h)+F_n(x-h) - 2 F_n(x)| \leq \min(4 \|F_n\|_\infty, h^2 \|F_n''\|_\infty) \leq \min(4 , h^2C^2 2^{2n}) \|F_n\|_\infty\]
  to obtain
\begin{multline}\label{eq:from_LP_decay_to_modulus} \sup_x \frac{|F(x+h)+F(x-h) - 2 F(x)|}{h} \leq \sum_{n \geq 0} \min(\frac{4}{h} , h C^2 2^{2n}) \|W_n \ast F\|_\infty.
\end{multline}

   On the other hand, for the natural duality $L_1(\R)^*=L_\infty(\R)$, $L_1^0(\R)$ is the annihilator in $L_1(\R)$ of the constants. Together with the weak-* compactness of the unit ball of $L_\infty$, this implies that a bounded net $(h_i)$ in $L_\infty(\R)$ has all its accumulation points constant if and only if $\lim_i \int h_i g=0$ for every $g \in L_1^0$. We can even take $g$ in a subset of $L_1^0$ spanning a dense subspace, for example the functions given by Lemma~\ref{lem:density}. Here we see that \eqref{item:w*cv} holds if and only if for every $x\in \R, h \in (0,1]$,
    \[ \lim_{r\to 0} \sup_y \int f(y+\frac{t}{r}) \frac{1}{h}(\chi_{[x,x+h]}(t) - \chi_{[x,x-h]}(t)) dt = 0,\]
    This is indeed equivalent to \eqref{eq:zygmund} because
    \[ \int f(y+\frac{t}{r}) \frac{1}{h}(\chi_{[x,x+h]}(t) - \chi_{[x,x-h]}(t)) dt = \frac{F(z+h') + F(z-h')-2F(z)}{h'}\]
    with $z=y+\frac{x}{r}$ and $h'=\frac{h}{r}$.
\end{proof}
We end this subsection with a standard fact that we will use later, and that generalizes the classical fact that $B_{\infty,\infty}^\alpha(\R)$ is the space of H\"older-continuous functions if $0<\alpha<1$.
\begin{lem}\label{lem:continuity_from_LittlewoodPaley} Let $\varphi:\R\to \R$ be measurable and locally integrable such that $\sum_n \|W_n \ast \varphi\|_\infty<\infty$. Then $\varphi$ is essentially bounded continuous, and for almost every $s,t$,
  \[ |\varphi(s)-\varphi(t)|\lesssim \sum_n \min(2^n|s-t|,1) \|W_n \ast \varphi\|_\infty. \]
\end{lem}
\begin{proof}
  $\varphi$ decomposes as $\sum_n W_n \ast \varphi$, and for $n\geq 0$
  \begin{align*} |(W_n \ast \varphi)(s) - (W_n \ast \varphi)(t)| &\leq \min(2 \|W_n \ast \varphi\|_\infty, \|(W_n \ast \varphi)'\|_\infty|s-t|)\\
    & \leq \min(2,2^n C |s-t|) \|W_n \ast \varphi\|_\infty
  \end{align*}
by Bernstein's inequality (as in Lemma~\ref{lem:Bernstein}).
%and $W'_n = W_n'\ast(W_{n-1}+W_n+W_{n-1})$, so the lemma follows from the bound
%\begin{align*} \|W_n' \ast \varphi\|_\infty & \leq  \|W_n'\|_1 \|(W_{n-1}+W_n+W_{n-1}) \ast \varphi\|_\infty\\ & \lesssim 2^n \sum_{n-1 \leq k \leq n+1} \|W_k \ast \varphi\|_\infty.\qedhere
%\end{align*}
\end{proof}
\subsection{Partitions of unity arguments}
If $f:\R\to \R_+$ is a locally integrable function, let $Mf$ denote Hardy-Littlewood maximal function
  \[ M f(x) = \sup_{t>0} \frac{1}{2t} \int_{-t}^t f(x-s) ds.\]
  We will need the easy inequality for $g \in \schwartz(\R)$:
  \begin{equation}\label{eq:HLmaximalfunction} \int |{g}(s)| f(x+s) ds  \leq \int |g'(t)| \int_{-|t|}^{|t|} f(x-s)ds dt \leq 2\|t g'\|_{L_1} M f(x).
  \end{equation}
  
\begin{prop}\label{prop:from_one_h_to_all} Let $(G_r)_{r>0}$ be a family of functions $G_r:\R\to \R_+$ satisfying the two conditions:
  \begin{itemize} \item for every $\lambda>0$, there is $K_\lambda$ such that $G_{\lambda r} \leq K_\lambda G_r$ for every $r>0$,
  \item there is $K>0$ such that $M G_r \leq K G_r$ for every $r>0$.
  \end{itemize}

  Then for every $h \in C^\infty_c(\R)\setminus\{0\}$, there is a constant $C=C(h,K,K_\lambda)$ such that the following holds: if $F:\R\to \R$ is a locally integrable function such that
  \[ \Big|\int F(x\pm\frac s r) \widehat h(s) ds\Big| \leq G_r(x)\]
  for every $r>0$ and $x \in \R$. Then for every $n \geq 1$,
  \[ |W_n \ast F(x)| \leq C G_{2^n}(x)\]
  and
  \[ |(W_n \ast F)'(x)| \leq 2^n C G_{2^n}(x).\]
\end{prop}
We will apply this proposition in two different situations. The first, obvious, is when $G_r(x)=A+B/r$. The second is
\begin{example}\label{ex:HardyLittlewood} For every $A,B >0$, the functions $G_r(x) = A+\frac{B}{(r|\sin x|)^{\frac 1 3}}$ satisfies the assumptions of Proposition~\ref{prop:from_one_h_to_all} with constants $K_\lambda$ and $K$ independent from $A,B$.
\end{example}
\begin{proof} The fact that $G_{\lambda r} \leq K_\lambda G_r$ holds with $K_\lambda=\max(\lambda^{1/3},\lambda^{-1/3})$ is clear. For $G_r$ to satisfy $M G_r \leq K G_r$ with constant $K$, it is enough to have $K\geq 1$ and that $Mf(x) \leq K f(x)$, for $f(x) = \frac{1}{|\sin x|^{\frac 1 3}}$. To justify this, by the symmetries of $f$ we can assume that $x \in [0,\pi/2]$. If $0<t\leq x/2$, we have $0<\sin(x/2)\leq \sin(x -s)$ for every $s \in [-t,t]$ and
  \[ \frac{1}{2t} \int_{-t}^t f(x-s) ds \leq f(x/2) =O( f(x)).\]
  If $x/2<t$, we have $[x-t,x+t]\subset [-3t,3t]$ so
  \[ \frac{1}{2t} \int_{-t}^t f(x-s) ds \leq \frac{1}{2t} \int_{-3t}^{3t} f(s) ds= O\Big(1+t^{-1/3}\Big) =O(f(x)).\qedhere\]
\end{proof}

The proof will use the following lemma.
\begin{lem}\label{lem:ideals_in_Cinftc} Let $E \subset C^\infty_c(\R)$ be a nonzero ideal (a vector subspace such that $f g \in E$ for every $f \in C^\infty_c(\R)$ and $g \in E$). If $E$ is stable by dilations $f \mapsto f(\lambda \cdot)$ for every $\lambda \in \R\setminus\{0\}$, then $E$ contains $C^\infty_c(\R\setminus\{0\})$.
\end{lem}
\begin{proof}
  $E$ contains a nonzero nonnegative function (namely $|f|^2= \overline{f}f$ for a nonzero $f \in E$), so using dilations and sums we see that for every compact subset $K \subset \R\setminus\{0\}$, $E$ contains a nonnegative function that is strictly positive on $K$. This implies that $E$ contains all $C^\infty$ functions supported in $K$.
\end{proof}

\begin{proof}[Proof of Proposition~\ref{prop:from_one_h_to_all}] The set $E$ of functions $h \in C^\infty_c(\R)$ for which there is a constant $C$ such that 
  \[ \Big|\int F(x\pm\frac s r) \widehat h(s) ds \Big| \leq C G_r(x)\]
  is clearly a vector space. It is nonzero by assumption. It is stable by dilations by factor $\lambda>0$ be the first hypothesis, and by dilations by factor $\lambda<0$ because the assumption is with $\pm s/r$. We claim that it is also stable by multiplication by $C^\infty_c(\R)$: if $h \in E$ and $g \in C^\infty_c(\R)$ is arbitrary,
  \begin{align*} \Big|\int F(x+\frac s r) \widehat{gh}(s) ds \Big| & = \Big|\iint \widehat{g}(t) \widehat{h}(s-t) F(x + \frac s r) ds dt\Big|\\
    & \leq \int |\widehat{g}(t)| CG_r(x+\frac t r) dt \leq C'G_r(x).
  \end{align*}
  The first inequality is the assumption $h\in E$, and the second \eqref{eq:HLmaximalfunction} with $C'=2C\|t\widehat{g}'\|_{L_1}$. By Lemma~\ref{lem:ideals_in_Cinftc}, we deduce that $E$ contains $C^\infty_c(\R^*)$. In particular, if $w$ is the function as in the definition of Besov spaces, $E$ contains the function $h_1:t\mapsto w(|t|)$ as well as $h_2:t\mapsto t w(|t|)$. This proves the lemma, because $W_n=2^n \widehat{h_1}(2^n \cdot)$ and $W_n'=-2i\pi 2^{2n} \widehat{h_2}(2^n \cdot) $, so
  \begin{align*}
    W_n \ast F(x) & = 2^n \int F(x-t) \widehat{h_1}(2^nt) dt = \int F(x-\frac{t}{2^n}) \widehat{h_1}(t) dt\\
    (W_n\ast F)'(x) & = -2i\pi 2^{2n}\int F(x-t) \widehat{h_2}(2^nt) dt = -2i\pi 2^n \int F(x-\frac{t}{2^n}) \widehat{h_2}(t) dt.\qedhere
  \end{align*}
\end{proof}

\section{From approximation properties of von Neumann algebras to regularity of Schur multipliers}\label{sec:AP_Schur}
This section is devoted to the proof of  Theorem~\ref{thm:rank0reductionImproved}. We first recall very briefly basic definitions required to make sense to the statement. Given a discrete group $\Gamma$, its von Neumann algebra $\LL\Gamma$ is the space of all bounded left-convolution operators on $\Gamma$, that is the set of bounded operators $a$ on $\ell_2(\Gamma)$ such that there is $\xi \in \ell_2(\Gamma)$ such that $a \eta = \xi\ast \eta$ for every $\eta \in \ell_2(\Gamma)$. It carries a tracial state $\tau_\Gamma(a) = \langle a\delta_e,\delta_e\rangle$. By the general construction (\S~\ref{subsection:NCLp}), for $1\leq p<\infty$, the map $a \mapsto \|a\|_{L_p(\tau_\Gamma)}:= (\tau_\Gamma (a^*a)^p\rangle^{\frac 1 p}$ is a norm on $\LL \Gamma$, whose completion is $L_p(\LL \Gamma)$, the non-commutative $L_p$ space of $\Gamma$. By the discussion in \cite[\S 3.1]{LafforguedlS}, the following is one of several equivalent possible definitions: $L_p(\LL \Gamma)$ has the operator space approximation property OAP if there is a net $T_\alpha$ of finite rank linear maps on $L_p(\LL \Gamma)$ such that 
\[ \lim_\alpha \sup_i \|T_\alpha\otimes \mathrm{id}(x_i) - x_i\|_{L_p(\tau_\Gamma\otimes \Tr)}=0\]
for every sequence $x_i \in L_p(\LL\Gamma) \otimes M_{n_i}$ such that $\lim_i \|x_i\|_{L_p(\tau_\Gamma \otimes \Tr)} = 0$. 

\begin{proof}[Proof of Theorem~\ref{thm:rank0reductionImproved}] We give the proof, insisting on the new aspects and going fast on the arguments that appear identically in the litterature. The reader can look at \cite{LafforguedlS,MR3781331,MR4680356} for more details.

  We recall a notion from \cite{LafforguedlS}: a locally compact group has the property of completely bounded approximation by Schur multipliers on $S_p$ if there is a constant $C$ and a net $\psi_\alpha \in A(G)$ such that $\psi_\alpha \to 1$ uniformly on compact subsets of $G$ and the Schur multiplier with symbol $\psi_\alpha(gh^{-1})$ has completely bounded norm less than $C$ on $S_p(L_2(G))$. The infimum of such constants is $\CHp(G)$, with the convention $\inf \emptyset= \infty$. Here $A(G) \subset C_0(G)$ is the Fourier algebra of $G$: \[A(G) = \{ \xi \ast \eta \mid \xi,\eta \in L_2(G).\]
  We prove that, under the assumption in Theorem~\ref{thm:rank0reductionImproved}, $\CHp(\SL_{2d-1}(\R))=\infty$. By the results of \cite{LafforguedlS,MR3035056,MR3781331} (see for example the proof of \cite[Theorem 4.7]{MR3781331}) it indeed implies that $L_p(\LL\Gamma)$ lacks the OAP for every $\Gamma$ a lattice in $\mathrm{SL}_{2d-1}(\R)$, or in a connected simple Lie group whose Lie algebra contains $\mathfrak{sl}_{2d-1}$. This includes all simple Lie groups of rank $2d-1$ if $2d-1=3$ or $2d-1\geq 9$. 

 Let us write $G=\mathrm{SL}_{2d-1}(\R)$ and $K=\mathrm{SO}(2d-1)$. For $v>0$, let us write $D(v) \in G$ the diagonal matrix with first $d-1$ entries equal to $v$, one equal to $0$ and last $d-1$ equal to $-v$. Let $\psi:G\to \C$ be a continuous $K$-biinvariant function vanishing at infinity such that the Schur multiplier with symbol $(g,h)\mapsto \psi(g^{-1}h)$ is bounded with norm $1$ on $S_p(L_2(G))$. We prove that
 \begin{equation}\label{eq:bound_on_psi} |\psi(D(v))| \leq C \sum_{n \geq 1} \varepsilon( \exp(-(1+c)^n v/C)).
 \end{equation}
 for some positive constants $c,C$ that depend on $d$ only. Before we do so, let us explain why \eqref{eq:bound_on_psi} proves the theorem. First, since $\varepsilon$ is non-decreasing, we have 
 \[ \varepsilon( \exp(-(1+c)^n v/C)) \leq \frac{1+c}{c} \int_{(1+c)^{n-1}}^{(1+c)^{n}} \frac{\varepsilon( \exp(-s v/C))}{s} ds.\]
 Summing over $n$, we obtain
 \[ |\psi(D(v))| \leq \frac{C+cC}{c} \int_{1}^\infty \frac{\varepsilon( \exp(-s v/C))}{s} ds = \frac{C+cC}{c} \int_0^{\exp(-v/C)} \frac{\varepsilon(t)}{t |\log t|} dt.\]
 The last equality is the change of variable $t=\exp(-s v/C)$. The right-hand side is finite and goes to $0$ as $v\to \infty$. In particular, for every real number $A>1$, there is $g \in G$ such that $|\psi(g)|\leq \frac 1 A$ for every $K$-biinvariant continuous function vanishing at infinity and with Schur multiplier norm $1$. Therefore, for every $\varphi \in A(G)$, we have $\int_{K\times K} \varphi(kgk')dk dk' \leq \frac{1}{A} \|\varphi\|_{MS_p(L_2(G))}$. This implies $\Lambda_{p,\mathrm{cb}}^{\mathrm{Schur}}(G) \geq A$, and therefore  $\Lambda_{p,\mathrm{cb}}^{\mathrm{Schur}}(G) = \infty$.

 It remains to prove \eqref{eq:bound_on_psi}. This is essentially what the argument in \cite{LafforguedlS} (for $d=2$) and \cite{MR3781331} (for $d>2$) proves. We first recall the argument when $d=2$. Let $r\geq s\geq t$ with $r+s+t=0$. Consider the function 
 \[ (k,k')\in K\times K \mapsto \psi( \diag(-t,t/2,t/2) k^{-1}k' \diag(-t,t/2,t/2)).\]
 Since $\diag(-t,t/2,t/2)$ commutes with $\SO(2)$, it descends to a function $\SO(3)/\SO(2) \times \SO(3)/\SO(2)$, so under the identification $\SO(3)/\SO(2) \simeq \sphere^{2}$, we can regard it as a function $\sphere^2\times \sphere^2\to\C$ of the form $(\xi,\eta)\mapsto \varphi(\langle \xi,\eta\rangle)$. By Lemma~\ref{lem:restriction}, it is a Schur multiplier with norm $\leq 1$ on $S_p(L_2(\sphere^2))$. Following these identifications, a small computation gives that for $\varphi(0) = \psi(\diag(-t/2,-t/2,t))$, and that for  $\delta  =\frac{\sinh(r+t/2)}{\sinh(-3t/2)}$, $\varphi(\delta) = \psi(\diag(r,s,t))$. So by the hypothesis we obtain
 \[ |\psi(\diag(r,s,t)) - \psi(\diag(-t/2,-t/2,t))| \leq \varepsilon(\delta) \leq \varepsilon(e^{t-s}).\]
 The second inequality is because $\varepsilon$ is non-decreasing and $\delta \leq e^{r+2t} = e^{t-s}$. Therefore if $r'\geq s' \geq t$ is another triple with $r'+s'+t=0$,
\[ |\psi(\diag(r,s,t)) - \psi(\diag(r',s',t))| \leq \varepsilon(e^{t-s}) + \varepsilon(e^{t'-s'}).\]
 By symmetry we also get
 \[ |\psi(\diag(r,s,t)) - \psi(\diag(r,s',t'))| \leq  \varepsilon(e^{s-r})+ \varepsilon(e^{s'-r'}),\]
 and combining the two, writing $g=\diag(3/2v,-v/2,-v)$ we obtain
 \begin{align*} |\psi(D(v)) - \psi(D(3v/2))| & \leq |\psi(D(v)) - \psi(g)| + |\psi(g)-\psi(D(3v/2))|\\
   & \leq 4 \varepsilon(\exp(-v/2)).
 \end{align*}
 %% As a consequence,
 %% \begin{align*}
 %%   |\psi(D(v))| &= \lim_{N \to \infty}  |\psi(D(v) -  \psi(D((3/2)^N v))|\\
 %%   & \leq \sum_{n=1}^\infty |\psi(D((3/2)^{n-1}v) -  \psi(D((3/2)^n v))|\\
 %%   & \leq 4 \sum_{\neq \geq 1} \varepsilon( (3/2)^{n-1}v/2).
 %% \end{align*}
 This proves \eqref{eq:bound_on_psi} with $c=\frac 1 2$ and $C=4$.

 When $d\geq 2$, following the proof of \cite[Lemma 4.3--4.4]{MR3781331} we obtain similarly
 \begin{multline*} |\psi(\diag(v_1,\dots,v_{d-1},u,t,\dots,t)) - \psi(\diag(v_1,\dots,v'_i,\dots,v_{d-1},u'-\delta,t,\dots,t))| \\\leq \varepsilon( e^{t-u}) + \varepsilon( e^{t-u'})
 \end{multline*}
 whenever both $(v_1,\dots,v_{d-1},u,t,\dots,t)$ and $(v_1,\dots,v'_i,\dots,v_{d-1},u'-\delta,t,\dots,t)$ belong to the Weyl chamber $\{a_1 \geq \dots \geq a_{2d-1}, \sum a_i=0\}$. By $d-1$ consecutive applications with $v_i=v$ and $v'_i = (1+1/d)v$, we obtain 
 \[ \big|\psi(D(v)) - \psi(\diag( \frac{d+1}{d}v,\dots,\frac{d+1}{d}v,-\frac{d-1}{d}v,-v,\dots,-v)\big|  \leq 2(d-1) \varepsilon(e^{-v/d})\]
 and by symmetry
\begin{multline*} \big|\psi(\diag( \frac{d+1}{d}v,\dots,\frac{d+1}{d}v,-\frac{d-1}{d}v,-v,\dots,-v) - \psi(D(\frac{d+1}{d}v)) \big|\\  \leq 2(d-1) \varepsilon(e^{-\frac{d+1}{d}v}).
\end{multline*}
Putting both estimates together, we obtain
   \[ |\psi(D(v)) - \psi(D(\frac{d+1}{d}v))| \leq 4(d-1) \varepsilon(e^{-v/d}),\]
   which proves \eqref{eq:bound_on_psi} with $c=\frac 1 d$ and $C=4(d-1)$.
 \end{proof}
We made the choice to follow the original argument from \cite{LafforguedlS}, which uses some rather delicate reasonings (\cite[\S 3]{LafforguedlS}). Alternatively, we can also use the more direct argument that originates from \cite{Vergara,dlS_habil,MR4680356} and argue that we can derive from the proof of \eqref{eq:bound_on_psi} an explicit separation \`a la Hahn-Banach between the identity operator and finite rank operators on $L_p(\LL \SL_{2d-1}(\Z))$. Indeed, if $s:\SL_{2d-1}(\R)/\SL_{2d-1}(\Z) \to \SL_{2d-1}(\R)$ is a measurable section and $m_v$ denotes the image of the uniform measure on $\SO(2d-1) \times \SO(2d-1) \times \SL_{2d-1}(\R)/\SL_{2d-1}(\Z)$ by the map
\[(k,k',\omega) \mapsto (kD(v) k' s(\omega))^{-1} s(k D(v) k'\omega) \in \Gamma,\] then the sequence of maps $T\in CB(L_p(\LL \SL_{2d-1}(\Z))) \mapsto \sum_\gamma m_v(\gamma) \tau (T(\lambda_\gamma)\lambda_\gamma^*)$ is Cauchy in the dual of the set of all completely bounded operators on $L_p(\LL \SL_{2d-1}(\Z))$. Its limit separates the identity (where it takes the value $1$) from the finite rank operators (where it vanishes).
\section{Qualitative regularity of radial Fourier multipliers}\label{sec:qualitative}
This section is devoted to the proof of Theorem~\ref{thm:mainqualitative}. The proof follows Fefferman's argument \cite{zbMATH03370942}. 
\subsection{Reduction to square function estimates}
In this section, we make the following assumption:  $m \colon (0,\infty) \to \R$ is a bounded measurable function, $d \geq 2$ and $p \neq 2$ are such that $T_{m(|\cdot|)}$ is bounded $L_p(\R^d)\to L_p(\R^d)$. We denote the Fourier multiplier $T_{m(|\cdot|)}$ by $T$, and its $L_p(\R^d)\to L_p(\R^d)$-norm by $\|T\|$. 

The unit ball of $L_\infty(\R)$, equipped with weak-* topology $\sigma(L_\infty(\R),L_1(\R))$, is metrizable and compact. Therefore, any bounded sequence in $L_\infty(\R)$ has a weak*-converging subsequence. This says that there are plenty of sequences $s_n,r_n$ satisfying the hypothesis of the next lemma.

\begin{lem}\label{lem:WOT-convergence} Let $s_n,r_n$ be a sequence of positive reals such that

  \begin{itemize}
  \item $\lim_n r_n = +\infty$ and
    \item $m(s_n\exp(\frac{\cdot}{r_n}))$ converges weak* to $m_\infty \in L_\infty(\R)$.
  \end{itemize}

  Then $m_\infty$ defines a Fourier multiplier on $L_p(\R)$ of norm $\leq \|T\|$.

  Moreover, for $u \in \sphere^{d-1}$, let $U_n$ be the linear isometry of $L_p(\R^d)$ given by $U_n f(x) = \exp(i\langle s_n u,x\rangle) (s_n/r_n)^{\frac{d}{p}}f(x s_n/r_n)$. Then $U_n^{-1} T U_n$ converges in the weak operator topology to the Fourier multiplier with symbol $m_\infty(\langle \cdot,u\rangle)$.
\end{lem}
\begin{proof} Let $f,g$ be two functions on $\R^d$ whose Fourier transforms are $C^\infty$ with compact support. In particular, $f \in L_p(\R^d)$ and $g \in L_q(\R^d)$, and the set of all such $f$ ($g$) form a dense subspace of $L_p(\R^d)$ (respectively $L_q(\R^d)$). We will prove that
\begin{equation}\label{eq:WOT_cv_on_dense_subspace}\lim_n \int U_n^{-1} T U_n f(x) \overline{g(x)} dx = \int m_\infty(\langle u,\xi\rangle) \widehat{f}(\xi) \widehat{\overline{g}}(\xi) d\xi.
\end{equation}
Since the sequence of operators $U_n^{-1} T U_n$ have all operator norm $\|T\|$, this will at the same time imply that $S$, the Fourier multiplier with symbol $m_\infty( \langle \cdot,u\rangle)$ is bounded on $L_p(\R^d)$ with norm $\leq \|T\|$ (which is equivalent to saying that the Fourier multiplier with symbol $m_\infty$ is bounded on $L_p(\R)$ with norm $\leq \|T\|$), and that $S$ is the limit of $U_n^{-1} T U_n$ in the weak-operator topology.

First observe that in the Schwartz space,
  \[ \widehat{U_n f}(\xi) = (s_n/r_n)^{\frac d q} \widehat{f}(r_n/s_n \xi - r_n u).\]
  This implies that, if $T_\varphi$ is a bounded Fourier multiplier with symbol $\varphi$, then  $U_n^{-1} T_\varphi U_n$ is the Fourier multiplier with symbol $\varphi(s_n (u + \frac{\cdot}{r_n}))$. In particular, $U_n^{-1} T U_n$ is the Fourier multiplier with symbol $m(|s_n(u + \frac{\cdot}{r_n})|)$. Therefore,
  \[ \int U_n^{-1} T U_n f(x) \overline{g(x)} dx = \int m(s_n|u+\frac{\xi}{r_n}|) h(\xi) d\xi,\]
  where $h(\xi) = \widehat{f}(\xi) \widehat{\overline{g}}(\xi)$ is a compactly supported $C^\infty$ function. In the decomposition $\R^d = \R u \oplus u^\perp$, this becomes
  \[ \int_{u^\perp} \int_\R m(s_n\sqrt{1+2\frac{t}{r_n} +  \frac{t^2+|\eta|^2}{r_n^2}}) h(tu + \eta) dt d\eta.\]
  In fact, we can restrict these intervals to $u^\perp \cap B(R)$ and $[-R,R]$ for some fixed $R$, because $h$ has compact support. For every fixed $\eta$, we can make the change of variable $t\to t'$ where $\exp(\frac{t'}{r_n})= \sqrt{1+2\frac{t}{r_n} +  \frac{t^2+|\eta|^2}{r_n^2}}$. This change of variable is not injective (there are two solutions in $t$ to the equation $e^{2t'/r_n}=1+2\frac{t}{r_n} +  \frac{t^2+|\eta|^2}{r_n^2}$), but it becomes injective in restriction to $[-R,R]$ as soon as $r_n >R$ (because the sum of the two solutions is $-2 r_n$). And moreover we have $t'=t+o(1)$ and $\frac{dt'}{dt} = 1+o(1)$. In particular, using that $h$ is continuous we have $h(tu+\eta)  dt = (1+o(1)) h(t'u + \eta) dt'$ uniformly and we obtain
  \[ \int U_n^{-1} T U_n f(x) \overline{g(x)} dx = \int_{u^\perp} \int_\R m(s_n\exp(\frac{t'}{r_n})) h(t'u +\eta) dt' d\eta + o(1).\]
  Taking the limit $n\to \infty$ and using the assumption of the weak*-convergence of $m(s_n \exp(\frac{\cdot}{r_n}))$ to $m_\infty$, we get that the limit in the left-hand side of \eqref{eq:WOT_cv_on_dense_subspace} exists and is equal to
  \[ \int_{u^\perp} \int_\R m_\infty(t) h(tu+\eta) dt d\eta.\]
  This is the right-hand side of \eqref{eq:WOT_cv_on_dense_subspace}.
\end{proof}

Denote by $M_\infty$ the set of all $m_\infty$ that appear as in Lemma~\ref{lem:WOT-convergence}, when $(s_n)_n$ and $(r_n)_n$ are sequences satisfying the hypothesis of Lemma~\ref{lem:WOT-convergence}.
\begin{example} If $m$ is the indicator function of the ball, then
  \[M_\infty = \{ \chi_{(-\infty,a)} \mid -\infty \leq a\leq \infty\}.\]

  More generally, if $m$ is a function that is continuous on $(0,\infty)\setminus\{1\}$, that has left- and right-limits $m_-$ and $m_+$ at $1$, then
  \[M_\infty = \{m(s) \mid s \in (0,\infty)\setminus\{1\}\} \cup \{ m_-\chi_{(-\infty,a)} + m_+ \chi_{(a,\infty)} \mid -\infty \leq a\leq \infty\}.\]
\end{example}
The following is a generalization of Meyer's lemma \cite{zbMATH03370942}.
\begin{lem}\label{lem:Meyer} For every $m_\infty \in M_\infty$, and every finite family $u_1,\dots,u_N \in \sphere^{d-1}$ and $f_1,\dots,f_N \in L_p(\R^d)$,
  \[ \Big\| \big(\sum_j |T_{m_\infty(\langle \cdot, u_j\rangle)} f_j|^2\big)^{\frac 1 2}\Big\|_p \leq \|T\| \Big\| \big(\sum_j |f_j|^2\big)^{\frac 1 2}\Big\|_p.\]
\end{lem}
\begin{proof} It is a classical fact \cite[Theorem 5.5.1]{MR2463316} that the boundedness of $T$ implies automatically a boundedness of its $\ell^2$-extension: for any finite sequence $g_1,\dots,g_N \in L_p(\R^d)$,
  \begin{equation}\label{eq:l2extension}\Big\| \big(\sum_j |T g_j|^2\big)^{\frac 1 2}\Big\|_p \leq \|T\| \Big\| \big(\sum_j |g_j|^2\big)^{\frac 1 2}\Big\|_p.
  \end{equation}
Let $(r_n),(s_n)$ be as in Lemma~\ref{lem:WOT-convergence}. Let $U_{n,j}$ be the isometry as in Lemma~\ref{lem:WOT-convergence} for this $s_n,r_n$ but for the vector $u_j$. We will apply~\eqref{eq:l2extension} to $g_j=U_{n,j}f_j$. What is important is that for different values of $j$, the operators $U_{n,j}$ differ by a phase multiplication only. Therefore,
  \[ \Big(\sum_j |g_j|^2\Big)^{\frac 1 2}(x) = \Big(\frac{s_n}{r_n}\Big)^{\frac d p}\big(\sum_j |f_j|^2\Big)^{\frac 1 2}\Big(\frac{s_n}{r_n} x\Big),\]
  and the right-hand side of \eqref{eq:l2extension} is equal to $\big\| \big(\sum_j |f_j|^2\big)^{\frac 1 2}\big\|_p$.
  For the same reason, the left-hand side of \eqref{eq:l2extension} is equal to
  \[\Big\| \big(\sum_j |U_{n,j}^{-1} T U_{n,j} f_j|^2\big)^{\frac 1 2}\Big\|_p.\]
  Lemma~\ref{lem:WOT-convergence} tells us that $U_{n,j}^{-1} T U_{n,j} f_j$ converges in the weak-operator topology to $T_{m_\infty(\langle \cdot,u_j\rangle)} f_j$ as $n\to \infty$. By the semi-continuity of the norm for the weak topology, the lemma follows.
\end{proof}

\subsection{Investigating directional square inequalities}
The following theorem characterizes the functions satisfying the conlusion of Lemma~\ref{lem:Meyer}.

\begin{thm}\label{thm:directional_square_functions_implies_constant}
  Let $m \in L_\infty(\R)$. Assume that there is $p>2$ and $d \geq 2$ such that for every finite family $u_1,\dots,u_N \in \sphere^{d-1}$ and $f_1,\dots,f_N \in L_p(\R^d)$,
\begin{equation}\label{eq:directional_square_functions} \Big\| \big(\sum_j |T_{m(\langle \cdot,u_j\rangle)} f_j|^2\big)^{\frac 1 2}\Big\|_p \leq \Big\| \big(\sum_j |f_j|^2\big)^{\frac 1 2}\Big\|_p.
\end{equation}
Then there is $c\in \C$ such that $m=c$ almost everywhere.
\end{thm}

The proof relies on two lemmas. The first is easy, il follows also from \cite{zbMATH03158261}.
\begin{lem}\label{lem:peetre} If $m \in L_\infty(\R)$ satisfies that for every closed interval $I \subset \R$ and every $f \in L_2(\R)$ supported in $I$, $T_m f$ is supported in $I$, then $m$ is almost everywhere constant.
\end{lem}
\begin{proof} The assumption is that $T_m$ commutes, in $B(L_2(\R))$, with the operator of multiplication by indicator function of intervals. Therefore, it commutes with all multiplication operators (the von Neumann algebra generated by the indicator functions of intervals is the whole $L_\infty(\R)$). In particular, $T_m$ commutes with the multiplication by $e^{ia\cdot }$ for every $a \in \R$, that is $m(\cdot -a) = m$ almost everywhere for every $a$. This implies that $m$ is almost everywhere constant.  
\end{proof}

In the second lemma, $m$ is as in Theorem~\ref{thm:directional_square_functions_implies_constant}. \begin{lem}\label{lem:supportTm} Let $f$ be a $C^\infty$ function supported in an interval $I=[s,t]$. Let $I_\pm$ the intervals $I\pm 2(t-s)$. Then
  \[ \int_{I_-} |T_mf|^2 = 0 \textrm{ and }\int_{I_-} |T_mf|^2=0.\] 
\end{lem}
The two lemmas imply Theorem~\ref{thm:directional_square_functions_implies_constant}: indeed, if $f$ is a $C^\infty$ function that is supported in an interval $I=[s,t]$, then it is also supported in the interval $I'=[s',t]$ for any $s'<s$, so the first statement in Lemma~\ref{lem:supportTm} implies that $\int_{I'_+} |T_mf|^2=0$ for every $s'<s$, that is $\int_{2t-s}^\infty |T_m f|^2 = 0$. Also, writing $f$ as a finite sum of $C^\infty$ functions supported in subintervals of $[s,t]$ of size $\leq \varepsilon$, we obtain $\int_{t+\varepsilon}^\infty |T_mf|^2=0$, so $\int_t^\infty |T_mf|^2=0$. Similarly, looking at the interval $[s,t']$ for $t'>t$, the second statement in Lemma~\ref{lem:supportTm} implies that $\int_{-\infty}^{s} |T_mf|^2=0$. So $T_m f$ is supported in $[s,t]$. But this is true for every $s<t$ and every $f$. Lemma~\ref{lem:peetre} implies that $m$ is a constant.

\begin{proof}[Proof of Lemma~\ref{lem:supportTm}] Using the dilation symmetries of $L_p(\R^d)$, we can assume that $t-s=1$. We have to prove $\int_{s+2}^{s+3}|T_m|^2=0$ and $\int_{s-2}^{s-1} |T_m f|^2=0$. We focus on proving the second. The other inequality can be proven similarly and is left to the reader.\footnote{The only difference is that we define $m_j(x) = m(\langle x,-u_j\rangle)$ instead of $m(\langle x,u_j\rangle)$.}

  The proof relies on the Besicovich construction (\S~\ref{subsec:besicovich}). Let $\varepsilon>0$. By the Besicovich's construction, there are $\delta>0$ and $\delta$-tubes $R_1,\dots,R_N$ satisfying \eqref{eq:gooddeltatubes}. That is, $E=\bigcup_{i=1}^N R_i$ has measure $\leq \varepsilon \sum_i |R_i|$. Write $\vecteur_j = \vecteur(R_j)$, and choose a linear isometry $x \in \R^d \mapsto (\langle x,\vecteur_j\rangle,P_j x) \in \R \times \R^{d-1}$. There are $a_j \in \R ,b_j \in \R^{d-1}$ such that
\[ R_j = \Big \{x \in \R^d \mid \langle x,\vecteur_j\rangle \in a_j + I \textrm{ and } |P_j x-b_j| \leq \delta\Big\}\]
\[ \overline{R_j} = \Big \{x \in \R^d \mid  \langle x,\vecteur_j\rangle \in a_j+2+I \textrm{ and } |P_j x-b_j| \leq \delta\Big\}.\]
Define $f_j:\R^d\to \R$ by
\[ f_j(x) = f\big(\langle x,\vecteur_j\rangle - a_j-2\big) h_j(P_j x),\]
  where $h_j$ is the indicator function of $B(b_j,\delta)$.  Then $f_j$ is supported in $\overline{R_j}$, which are pairwise disjoint so
  \[ \Big\| \big(\sum_j |f_j|^2\big)^{\frac 1 2}\Big\|_p = (\sum_j \|f_j\|_p^p)^{\frac 1 p} = (\sum_j |R_j|)^{\frac 1 p}\|f\|_p.\]
  On the other hand, if $m_j(x)= m(\langle x,u_j\rangle)$, since in the decomposition $\R^d = \R u_j \oplus \R^{d-1}$, $f_j$ decomposes as $f_j(\cdot - a_j-2)\otimes h_j$ and $m_j=m\otimes 1$, we have $T_{m_j} f_j(x) = (T_mf)( \langle x,u_j\rangle-a_j-2) h_j(P_j x)$. We can therefore bound
  \begin{align*} \int_E \sum_j |T_{m_j}f_j|^2 & \geq \sum_j \int_{R_j} |T_m(\langle x,u_j\rangle - a_j-2)|^2 |h_j(P_jx)|^2 dx\\
    & = \sum_j |R_j|\int_{s-2}^{s-1} |T_mf|^2.
  \end{align*}
By H\"older's inequality
    \begin{align*} \Big(\int_E \sum_j |T_{m_j}f_j|^2\Big)^{\frac 1 2} & \leq |E|^{\frac 1 2 - \frac 1 p} \Big\| \big(\sum_j |T_{m_j} f_j|^2\big)^{\frac 1 2}\Big\|_p\\
    & \leq \varepsilon^{\frac 1 2 -\frac 1 p}  (\sum_j |R_j|)^{\frac 1 2} \|f\|_p.
    \end{align*}
Putting everything together we get
\[  \Big(\int_{I_-} |T_mf|^2\Big)^{\frac 1 2} \leq  \varepsilon^{\frac 1 2 - \frac 1 p} \|f\|_p.\]
Taking the infimum over all choices of $R_1,\dots,R_N$, we can replace $\varepsilon$ by $\fonction_d(\delta)$ in this inequality, and since $\lim_{\delta \to 0} \fonction_d(\delta)=0$, we get $\Big( \int_{I_-} |T_mf|^2\Big)^{\frac 1 2}=0$ as required.
\end{proof}

\subsection{The conclusion}

%% \begin{cor} Let $p\neq 2$ and $d\geq 2$. Let $T=T_{m(|\cdot|)}$ be an $L_p(\R^d)$ bounded radial multiplier. Then the function $m\circ \exp$ belongs to $b_{\infty,\infty}^0(\R)$.
%% \end{cor}
\begin{proof}[Proof of Theorem~\ref{thm:mainqualitative}]
  Write $f=m\circ \exp$. Let $M_\infty$ denote all $\sigma(L_\infty(\R),L_1(\R))$-limits of sequence $m(s_n e^{\frac{t}{r_n}}) = f(\log(s_n) + \frac{\cdot}{r_n})$ for $s_n \in (0,\infty)$ and $r_n \to 0$. All $m_\infty \in M_\infty$ satisfy the conclusion of Lemma~\ref{lem:Meyer}, so are almost everywhere constant by Theorem~\ref{thm:directional_square_functions_implies_constant}. Proposition~\ref{prop:characterization_of_binftyinfty0} implies that $f \in b_{\infty,\infty}^0$, or equivalently that $F$ is smooth.
\end{proof}

\section{Quantitative regularity of radial Fourier multipliers}\label{sec:quantitative}
This section is devoted to the proof of Theorem~\ref{thm:Lpbddness_LP_quantitative} and its corollaries. The proof sometimes uses and follows arguments in the previous section, and sometimes deviates from it.

Theorem~\ref{thm:Lpbddness_LP_quantitative} will easily follow from the following lemma.
\begin{lem}\label{lem:Lpbddness_quantitative} There is a nonzero $h \in C^\infty_c(\R)$ and a constant $C_d$ such that, for every $1 \leq p \leq \infty$, every $\varphi$ as in Theorem~\ref{thm:Lpbddness_LP_quantitative}, every $r \in \R^*$ and $\delta \in (0,1)$,
  \[ \big| \int_\R \varphi(x/r) \widehat h(x) dx\big| \leq C_d \|T\|_{L_p\to L_p}\big( \fonction_d(\delta)^{\big|\frac 1 p - \frac 1 2\big|} + \frac{1}{|r| \delta^2}\big).\]
\end{lem}
\begin{proof}[Proof of Theorem~\ref{thm:Lpbddness_LP_quantitative} assuming Lemma~\ref{lem:Lpbddness_quantitative}] The first part of the theorem is  an application of Proposition~\ref{prop:from_one_h_to_all} with
  \[ G_r(x) = C_d \|T\|_{L_p\to L_p}\big( \fonction_d(\delta)^{\big|\frac 1 p - \frac 1 2\big|} + \frac{1}{|r| \delta^2}\big).\]
The assumptions of the proposition are immediate to check: $G_r$ is a constant function so $MG_r = G_r$ is obvious, and $G_{\lambda r} \leq \max(\lambda,1/\lambda) G_r$ is clear.

The second part is by Keich's construction of Besicovich sets \eqref{eq:Keich} and Lemma~\ref{lem:infi_for_particular_fd} below. The last statement follows by standard Littlewood-Paley theory \eqref{eq:from_LP_decay_to_modulus}.\end{proof}
\begin{lem}\label{lem:infi_for_particular_fd} For every $\alpha,\beta>0$, and $A \geq 1$,
  \[ \inf_{0<\delta} \delta^{\alpha} + \frac{1}{A \delta^\beta} \leq 2 A^{-\frac{\alpha}{\alpha+\beta}}.\]

\[ \inf_{0<\delta<1} |\log \delta|^{-\alpha} + \frac{1}{A\delta^{\beta}} \leq C(\alpha,\beta) (\log (1+A))^{-\alpha}. \]
\end{lem}
\begin{proof}
  For the first inequality, take $\delta = A^{-\frac{1}{\alpha + \beta}}$, so both terms $\delta^{\alpha}$ and $\frac{1}{A \delta^\beta}$ are equal to $A^{-\frac{\alpha}{\alpha+\beta}}$.

  For the second, take $\delta = \min(A^{-\frac{1}{\beta}}\log(1+A)^{\frac{\alpha}{\beta}},1/2)$. Both terms are $O(\log (1+A)^{-\alpha})$.
\end{proof}

\subsection{Proof of Lemma~\ref{lem:Lpbddness_quantitative}}
We can  normalize $m$ so that $\|T\|_{L_p \to L_p}=1$. We can also assume that $p<2$, otherwise we replace $p$ by its conjugate exponent $q = \frac{p}{p-1}$. Finally, we can assume $|r| \geq 1$, because otherwise the lemma is obvious for any $h$ by the bound 
\[ \big| \int_\R \varphi(x/r) \widehat h(x) dx\big| \leq \|\varphi\|_\infty \|\widehat h\|_1.\]

Let $f,g \in C^\infty_c(\R)$ supported in $(-1,0)$ and $(1,2)$ respectively. We will prove the lemma with $h=f\ast g^*$, where $g^*(x) = \overline{g(-x)}$.

Let $\varepsilon>0$ and $\{R_j\}$ be a family of $\delta$-tubes such that the $\overline{R_j}$ are pairwise disjoint and that satisfy \eqref{eq:gooddeltatubes}. Write $\vecteur_j = \vecteur(R_j)$, and choose a linear isometry $x \in \R^d \mapsto (\langle x,\vecteur_j\rangle,P_j x) \in \R \times \R^{d-1}$. There are $a_j \in \R ,b_j \in \R^{d-1}$ such that
\[ R_j = \Big \{x \in \R^d \mid \langle x,\vecteur_j\rangle \in [a_j-1,a_j] \textrm{ and } |P_j x-b_j| \leq \delta\Big\}\]
  \[ \overline{R_j} = \Big \{x \in \R^d \mid  \langle x,\vecteur_j\rangle \in [a_j+1,a_j+2] \textrm{ and } |P_j x-b_j| \leq \delta\Big\}.\]

Let us fix $\rho \in C^\infty_c(\R^{d-1})$ supported inside $B(0,1)$ and with $\|\rho\|_2=1$. Let us define $f_j,g_j:\R^d\to \R$ by
  \begin{equation}\label{eq:def_of_fj} f_j(x) = \delta^{-\frac{d-1}{p}}f\big(\langle x,\vecteur_j\rangle - a_j\big) \rho\big( (P_j x - b_j)/\delta\big),
  \end{equation}
  \begin{equation}\label{eq:def_of_gj} g_j(x) = \delta^{-\frac{d-1}{p}} g\big(\langle x,\vecteur_j\rangle - a_j\big) \rho\big( (P_j x-b_j)/\delta\big),
  \end{equation}
  so that $f_j$ is supported in $R_j$, $g_j$ is supported in $\overline{R_j}$, $\|f_j\|_p =\|f\|_p \|\rho\|_p$ and $\|g_j\|_q = \|g\|_q \|\rho\|_q$. Finally, let us set $\varphi_j(x) = m(|u_j+\frac{1}{r} x|)$ and $T_j$ the Fourier multiplier with symbol $\varphi_j$. All these are defined so that $\int (T_j f_j) \overline{g_j}$ does not depend on $j$ and satisfies the following (recall $h=f \ast g^*$).
  \begin{lem}\label{lem:technical} There is a constant $C=C(d,\rho,f,g)$ such that or every $j$ and $|r|\geq 1$,
    \[ \Big|\int_{\R^d} (T_j f_j) \overline{g_j} - \int_\R \varphi(\cdot /r) \widehat{h}\Big| \leq C \frac{1}{|r|\delta^2}.\]
  \end{lem}
  \begin{rem} The statement does not need any hypothesis on the support of $f,g$ and the constant is of the form
    \[ C(d,\rho,f,g) = C(d,\rho) \Big(\frac{1}{r\delta^2} \|\widehat{h}'\|_1 + \frac{1}{r}  \|x^2\widehat{h}'\|_1\Big).\]
    \end{rem} 
  Before we prove this lemma, let us conclude the proof of Lemma~\ref{lem:Lpbddness_quantitative}. We shall need the following fact, that is essentially Meyer's Lemma from \cite{zbMATH03370942}.
  \begin{lem}\label{lem:square_function_inequality} Let $r\neq 0$, $N \in \N$ and $u_1,\dots,u_N \in \sphere^{d-1}$. Let $T_j$ denote the Fourier multiplier with symbol $x\mapsto m(|u_j+\frac 1 r x|)$. Then for every $f_1,\dots,f_N \in L_p(\R^d)$,
  \[ \Big\| \big(\sum_j |T_{j} f_j|^2\big)^{\frac 1 2}\Big\|_p \leq \Big\| \big(\sum_j |f_j|^2\big)^{\frac 1 2}\Big\|_p.\]
\end{lem}
\begin{proof}  This is what the arguments in the proof of Lemma~\ref{lem:WOT-convergence} and \ref{lem:Meyer} prove.
\end{proof}
Next we exploit the fact that the $\delta$-tubes $\{R_j\}$ satisfy \eqref{eq:gooddeltatubes} as follows.
\begin{lem}\label{lem:consequence_Kakeya} Let $c$ denote the volume of the unit ball in $\R^{d-1}$, then
  \[ \Big\|(\sum_{j=1}^N |f_j|^2)^{\frac 1 2}\Big\|_p \Big\| (\sum_{j=1}^N |g_j|^2)^{\frac 1 2}\|_q \leq (c \varepsilon)^{\frac 1 p - \frac 1 2} N \|f\|_2 \|g\|_q \|\rho\|_q.\]
\end{lem}
\begin{proof} Since $g_j$ is supported in $\overline{R_j}$ and the $\overline{R_j}$ are disjoint, we have
  \[ \Big\|(\sum_j |g_j|^2)^{\frac 1 2}\Big\|_q = (\sum_j \|g_j\|_q^q)^{\frac 1 q} = N^{\frac 1 q} \|g\|_q \|\rho\|_q.\]
  By Hölder's inequality and the fact that $\sum_j |f_j|^2$ is supported inside $\bigcup_j R_j$ which has measure $\leq c \varepsilon  N \delta^{d-1}$, we have
  \[ \Big\|(\sum_j |f_j|^2)^{\frac 1 2}\Big\|_p \leq (c \varepsilon  N \delta^{d-1})^{\frac 1 p - \frac 1 2} \Big\|(\sum_j |f_j|^2)^{\frac 1 2}\Big\|_2 = (c\varepsilon)^{\frac 1 p - \frac 1 2} N^{\frac 1 p} \|f\|_2 .\]
  The lemma follows.
\end{proof}
By the Cauchy-Schwarz and Hölder inequalities, we can bound
  \begin{align*} \Big|\sum_j \int (T_j f_j) \overline{g_j}\Big| & \leq \int (\sum_j |T_j f_j|^2)^{\frac 1 2} (\sum_j |g_j|^2)^{\frac 1 2}\\
    & \leq \Big\|(\sum_j |T_j f_j|^2)^{\frac 1 2}\Big\|_p \Big\| (\sum_j |g_j|^2)^{\frac 1 2}\Big\|_q\\
    & \leq \Big\|(\sum_j |f_j|^2)^{\frac 1 2}\Big\|_p \Big\| (\sum_j |g_j|^2)^{\frac 1 2}\Big\|_q\\
    & \leq N (c\varepsilon)^{\frac 1 p - \frac 1 2} \|f\|_2 \|g\|_q \|\rho\|_q.
  \end{align*}
  The last two lines are Lemma~\ref{lem:square_function_inequality} and Lemma~\ref{lem:consequence_Kakeya}.
  Taking into account Lemma~\ref{lem:technical}, we get
  \[ |\int \varphi(x/r) \widehat h(x) dx| \lesssim \frac{1}{|r|\delta^2} + \varepsilon^{\frac 1 p - \frac 1 2}.\]
    It remains to take the infimum over all $R_i$ to obtain the conclusion of Lemma~\ref{lem:Lpbddness_quantitative}.

We are left to prove Lemma~\ref{lem:technical}. The proof is an elaboration on the soft arguments in Lemma~\ref{lem:WOT-convergence}.
  \begin{proof}[Proof of Lemma~\ref{lem:technical}]
By a change of variable in the integral formula for the Fourier transform, we have
  \begin{equation*}\label{eq:FT_of_f_j} \widehat{f_j}(x) = \delta^{\frac{d-1}{q}}e^{-2i\pi(a_j \langle x,\vecteur_j\rangle + \langle P_j x,b_j\rangle)} \widehat f(\langle x,\vecteur_j\rangle) \widehat{\rho}(\delta P_j x) 
  \end{equation*}
  and
  \begin{equation*}\label{eq:FT_of_g_j} \widehat{g_j}(x) = \delta^{\frac{d-1}{p}}e^{-2i\pi(a_j \langle x,\vecteur_j\rangle + \langle P_j x,b_j\rangle)} \widehat g(\langle x,\vecteur_j\rangle) \widehat{\rho}(\delta P_j x).
  \end{equation*}
Therefore, since $h$ was defined so that $\widehat h=\widehat f\overline{\widehat g}$, 
  \begin{equation}\label{eq:FT_of_f_jg_j}\widehat f_j(x)\overline{ \widehat{g_j}(x)} = \delta^{d-1} \widehat{h}(\langle x,\vecteur_j\rangle) |\widehat{\rho}|^2(\delta P_j x).
  \end{equation}
  By the definition of the Fourier multiplier $T_j$, we therefore have
  \begin{align*} \int_{\R^d} (T_j f_j) \overline{g_j} &= \int_{\R^d} m(|u_j+\frac{1}{r} x|) \widehat{f_j}(x) \overline{\widehat g_j}(x) dx\\
    & =\delta^{d-1}\iint m(\sqrt{\big(1+\frac x r)^2+\frac{|y|^2}{r^2}}) \widehat{h}(x)  |\widehat \rho(\delta y)|^2 dx dy\\
& =\iint m(\sqrt{\big(1+\frac x r)^2+\frac{|y|^2}{\delta^{2}r^2}}) \widehat{h}(x) |\widehat \rho(y)|^2 dx dy.
  \end{align*}
Assume first  $r\geq 1$. We are in the setting of Lemma~\ref{lem:change_of_variable} with $F:=\widehat h$, $G:=|\widehat \rho|^2$, $r_1=r$, $r_2=\delta r$, so that $\frac{r_1}{r_2^2}=\frac{1}{\delta^2 r}$. The functions $F$ and $G$ being Schwartz functions that depend in $d,f,g,\rho$, we deduce Lemma~\ref{lem:technical} as a direct application of Lemma~\ref{lem:change_of_variable}.

If $r \leq -1$, we can make the change of variable $x\to -x$ and get
\[ \int_{\R^d} (T_j f_j) \overline{g_j} = \iint m(\sqrt{\big(1+\frac x {|r|})^2+\frac{|y|^2}{\delta^{2}r^2}}) \widehat{h}(-x) |\widehat \rho(y)|^2 dx dy\]
and the conclusion follows identically.
  \end{proof}
To complete the proof of Lemma~\ref{lem:Lpbddness_quantitative}, it remains to prove the next lemma that was used above.
\begin{lem}\label{lem:change_of_variable} Let $F \in \schwartz(\R)$, $G \in \schwartz(\R^{d-1})$, $m \in L_\infty(\R)$ of norm $1$ and $r_1 \geq r_2\geq 1$. Set
  \begin{align*} A &= \iint F(x) G(y) m\Big(\sqrt{(1+\frac{x}{r_1})^2+\frac{|y|^2}{r_2^2}}\Big) dx dy\\
    B &= \Big(\int F(x)  m(\exp(\frac x {r_1})) dx \Big)\Big(\int G(y)dy\Big)
  \end{align*}
  Then
  \[ |A - B |\lesssim \|(1+|y|^2) G\|_1  \Big(\frac{r_1}{r_2^2} \|F'\|_1 + \frac{1}{r_1} \|x^2 F'\|_1\Big)\]
\end{lem}
\begin{proof}The idea of the proof is straightforward: in the regime when $r_1=o(r_2^2)$ and $r_1\to \infty$ with $x,y$ bounded, $\sqrt{(1+\frac{x}{r_1})^2+\frac{y^2}{r_2^2}}$ is $\exp(\frac{x}{r_1} + o(\frac 1 {r_1}))$. Of course, since we do not have any regularity hypothesis on $m$, we cannot directly bound $A-B$ in these coordinates, so we make a change of variables to be able to use the regularity of $F$.

  Now the details. For convenience we normalize $G$ so that $\|(1+|y|^2) G\|_1 = 1$. Let $\alpha \leq \frac 1 2$, the precise value of which will be determined along the proof. We implicitely make sure that every $O(\cdot)$ that we write is independent of $\alpha$. Consider
\[ \Omega= \Big\{(x,y) \in \R\times \R^{d-1} \mid |x|\leq \alpha r_1, |y|\leq r_2\Big\}\]
  and
\[A_1:= \iint_\Omega F(x) G(y) m\Big(\sqrt{(1+\frac{x}{r_1})^2+\frac{|y|^2}{r_2^2}}\Big) dx dy.\]
  We have
\begin{align*}  |A-A_1|& \leq \iint_{\R^d\setminus \Omega} |F(x) G(y)| dx dy\\& \leq \|F\|_{L_1([-\alpha r_1,\alpha r_1]^c)}\|G\|_1+\|F\|_1 \|G\|_{L_1(B(r_2)^c)},
\end{align*}
that we simply bound (recall $\|(1+|y|^2)G\|_1 =1$) as
\begin{equation}\label{eq:AA1} |A-A_1| \leq \frac{1}{\alpha r_1} \|x F\|_{1}+\frac{1}{r_2^2} \|F\|_1.
\end{equation}
Consider the change of variables $(x,y) \in \Omega \mapsto (u,y)$ where $u = u(x,y) \in \R$ is characterized by $e^{\frac u{r_1}} = \sqrt{(1+\frac{x}{r_1})^2+\frac{|y|^2}{r_2^2}}$. It is injective on $\Omega$, its image is
  \[ \Sigma = \Big \{(u,y) \in \R^d \mid |y|\leq  r_2, \exp(\frac{2u}{r_1}) - \frac{|y|^2}{r_2^2} \in [(1-\alpha)^2,(1+\alpha)^2]\Big\}\]
  and the inverse is $(u,y)\mapsto (x(u,y),y)$ where $x(u,y)=r_1\sqrt{\exp(\frac{2u}{r_1}) - \frac{|y|^2}{r_2^2}} - r_1$. Observe for further use that for $(x,y) \in \Omega$, $e^{\frac{2u}{r_1}} \in [ (1-\alpha)^2,(1+\alpha)^2+1] \subset [1/4,3]$, so $\frac{u}{r_1} = O ( e^{\frac{2u}{r_1}}-1)$ and 
  \begin{equation}\label{eq:apriori_bound_u} |u(x,y)| = O(|x| + \frac{|y|^2r_1}{r_2^2}).
  \end{equation}
Performing the change of variable in the integral, we obtain
  \[ A_1 = \iint_\Sigma F(x(u,y)) G(y) m(\exp(\frac u{r_1})) \partial_u x(u,y) du dy.\]
  We will see that $A_1$ is close to
  \[B_1:=\iint_\Sigma F(u) G(y) m(\exp(\frac u {r_1})) du dy.\]
  First observe that $\partial_u x(u,y) = \frac{\exp(2u/r_1)}{\sqrt{\exp(2u/r_1) - |y|^2/r_2^{2}}}$ is close to $1$: 
  \[\Big|\partial_u x(u,y) - 1 \Big| =O\Big( \frac{|u|}{r_1} + \frac{|y|^2}{r_2^2}\Big)\leq C \partial_u x(u,y) \big(\frac{|x|}{r_1} + \frac{|y|^2}{r_2^2}\big).\]
The second inequality is \eqref{eq:apriori_bound_u}, because on $\Sigma$, $\partial_u x(u,x)$ is bounded below. By the same change of variables, it follows that
  \[ A_2 := \iint_\Sigma F(x(u,y)) G(y) m(\exp(\frac u{r_1}))du dy\]
  satisfies
  \begin{equation}\label{eq:A1A2} |A_1-A_2|\leq C \iint_\Omega |F(x) G(y)| \big(\frac{|x|}{r_1} + \frac{|y|^2}{r_2^2}\big)dx dy \leq C \Big(\frac{1}{r_1} \|xF\|_1 + \frac{1}{r_2^2} \|F\|_1\Big).
  \end{equation}
  Then observe that, on $\Sigma$, $x(u,y)$ is close to $u$:
  \[ |x(u,y)-u| = O\big(\frac{u^2}{r_1} + \frac{r_1 |y|^2}{r_2^2}\big).\]
  We can therefore bound
  \[ |F(x(u,y)) - F(u)| \leq \int_{-C'\big(\frac{u^2}{r_1} + \frac{r_1 |y|^2}{r_2^2}\big)}^{C'\big(\frac{u^2}{r_1} + \frac{r_1 |y|^2}{r_2^2}\big)} |F'(u+s)| ,\]
%  \[ |F(x(u,y)) - F(u)| \leq \int |F'(u+s)| 1_{|s| \leq \big[-C'' \big(\frac{u^2}{r_1} + \frac{r_1 y^2}{r_2^2}\big),C''\big(\frac{u^2}{r_1} + \frac{r_1 y^2}{r_2^2}\big)\big]} ds,\]
    which after integration and change of variable $s\mapsto v=u+s$ leads to
    \[ |A_2-B_1| \leq \iint_\Sigma \int |F'(v) G(y)| 1_{|v-u| \leq C'\big(\frac{u^2}{r_1} + \frac{r_1 |y|^2}{r_2^2}\big)} dv du dy.\]
    Let $\beta \leq  1$ be some parameter that we will fix shortly. Partition $\Sigma$ as $\Sigma_1\cup \Sigma_2$ where $\Sigma_1=\{(u,y)\in \Sigma \mid  \beta |u|/r_1 \leq |y|/r_2\}$ and $\Sigma_2=\Sigma \setminus \Sigma_1$. Then we have
    \begin{multline*} |A_1 - B_1|\leq \iint_{\Sigma_1} \int |F'(v) G(y)| 1_{|v-u| \leq C'(1+\beta^{-2}) \frac{r_1 |y|^2}{r_2^2}} dv du dy\\ + \iint_{\Sigma_2} \int |F'(v) G(y)| 1_{|v-u| \leq 2 C'\frac{u^2}{r_1}} dv du dy.
    \end{multline*}
The first integral is obviously bounded by
\[ \int_{y \in \R^d} \int_{v \in \R} |F'(v) G(y)| 2 C'(1+\beta^{-2})\frac{r_1 |y|^2}{r_2^2} dv dy = 2C'(1+\beta^{-2}) \frac{r_1}{r_2^2} \|F'\|_1 \|G\|_1.\]
  For the second interval, we need that $\alpha$ and $\beta$ are small enough.  Indeed, on $\Sigma_2$ we know from \eqref{eq:apriori_bound_u} and the fact that $|u|/r_1=O(1)$, that
  \[ \frac{u}{r_1} = O(\alpha + \beta),\]
  so if $\alpha$ and $\beta$ are chosen to be small enough, we have that on $\Sigma_2$, $2C' \frac{|u|}{r_1} \leq \frac 1 2$. This implies that, if $|v-u| \leq 2 C'\frac{u^2}{r_1}$, we have $|v-u|\leq |u|/2$, so $|u|\leq 2|v|$, and finally $|v-u| \leq 8C'\frac{v^2}{r_1}$, so the second integral is less than
  \[ \int_{\R^{d-1}} \int_\R \int_\R |F'(v) G(y)| 1_{|v-u| \leq 8 C'\frac{v^2}{r_1}} dv du dy = \frac{16C'}{r_1} \|G\|_1 \|v^2 F'\|_1.\]
  Combining the two estimates, we get (with now $x$ the integration variable)
  \begin{equation}\label{eq:A2B1} |A_2-B_1| \lesssim  \frac{r_1}{r_2^2} \|F'\|_1 + \frac{1}{r_1} \|x^2 F'\|_1.
    \end{equation}
  We finally move to bound $|B-B_1|$. Again, we need that $\alpha$ is small enough so that it satisfies $e^{2\alpha^2}\leq (1+\alpha)^2$ and $e^{-2\alpha^2} - \alpha \geq (1-\alpha)^2$. Indeed, then every $(u,y)$ with $|u|\leq \alpha^2r_1$ and $\frac{|y|^2}{r_2^2} \leq \alpha$ satisfies $e^{\frac{2u}{r_1}}-\frac{|y|^2}{r_2^2} \in [e^{-2\alpha^2}-\alpha,e^{2\alpha^2}] \subset [(1-\alpha)^2,(1+\alpha)^2)]$, that is $(u,y) \in \Sigma$. By the contrapositive, a pair $(u,y)\notin \Sigma$ has to satisfy $|u|> \alpha^2 r_1$ or $|y| \geq r_2  \sqrt{\alpha}$. As a consequence, we have
    \[|B-B_1|\leq \|F\|_{L_1([-\alpha^2 r_1,\alpha^2 r_1]^c)} \|G\|_1 + \|F\|_1 \|G\|_{L_1([-r_2 \sqrt{\alpha},r_2 \sqrt{\alpha}]^c)},\]
    that we simply bound by
    \begin{equation}\label{eq:BB1} |B-B_1|\leq \frac{1}{\alpha^2 r_1}\|xF\|_1 + \frac{1}{\alpha r_2^2}\|F\|_1.\end{equation}
    Combining all the bounds, we see that there is a choice of $\alpha$ that makes all bounds \eqref{eq:AA1}, \eqref{eq:A1A2}, \eqref{eq:A2B1} and \eqref{eq:BB1} valid. This proves that
\[ |A-B| \lesssim \frac{1}{r_1} \|xF\|_1+\frac{1}{r_2^2} \|F\|_1 + \frac{r_1}{r_2^2} \|F'\|_1 + \frac{1}{r_1}\|x^2 F'\|_1.\]
To obtain the conclusion of the lemma, just use the following general inequalities valid for Schwartz functions on $\R$ and real numbers $r_2 \geq 1$:
  \begin{equation}\label{eq:L1xFL1xxFprime} \|x F\|_1 \leq \frac 1 2 \|x^2F'\|_1,
  \end{equation}
  \begin{equation}\label{eq:L1FL1Fprimesuboptimal} \frac{1}{r_2^2}\|F\|_1 \leq \frac{r_1}{r_2^2} \|F'\|_1 + \frac{1}{r_1}\|x^2 F'\|_1.
  \end{equation}
  To prove \eqref{eq:L1xFL1xxFprime}, compute
  \begin{align*} \int_{-\infty}^0 |xF(x)|dx & = \int_{-\infty}^0 |x| \Big|\int_{-\infty}^x F'(t) dt\Big|dx\\
    & \leq \int_{-\infty}^0 |F'(t)| \Big|\int_{t}^0 |x| dx\Big| dt = \frac{1}{2} \int_{-\infty}^0 |F'(t)| t^2 dt.
  \end{align*} Similarly, $\int_0^{\infty} |xF(x)|dx \leq \int_0^\infty t^2 |F'(t)|dt$. For \eqref{eq:L1FL1Fprimesuboptimal} we can start with the same proof to get
  \[ \|F\|_1 \leq \|xF'\|_1,\]
  which by the inequality $|x| \leq r_1+x^2/r_1$ gives $\|F\|_1 \leq r_1 \|F'\|_1 + \frac{1}{r_1} \|x^2 F'\|_1$. This implies \eqref{eq:L1FL1Fprimesuboptimal} (recall $r_2 \geq 1$).
\end{proof}
\begin{rem} Assume $1 \leq p<2$ and let $q=\frac{p}{p-1}$ be the conjugate exponent. More precisely, the proof shows that whenever $h=f\ast g^*$ with $f$ supported in $[0,1]$ and $g$ supported in $[a,b]$,%% \footnote{Or $[2,\infty]$ if we request that in the definition of $\fonction_d(\delta)$ that the sets $\bigcup_{n \geq 2} R_i+n\vecteur(R_i)$ are pairwise disjoint}
  then the conclusion of Lemma~\ref{lem:Lpbddness_quantitative} holds with
  \[ C_d(h) \lesssim_d \frac{1}{r\delta^2} \|\widehat{h}'\|_1 + \frac{1}{r}  \|x^2\widehat{h}'\|_1 + \fonction_d^{[a,b]}(\delta)^{\frac 1 p - \frac 1 2} \|f\|_2 \|g\|_q.\]
\end{rem}
\subsection{Consequences}
The following is a precise form of Corollary~\ref{cor:Lpbddness_LP_quantitative}. The exponent $\frac 1 4$ is not optimized.
\begin{cor}\label{cor:Lpbddness_LP_quantitative_precise}
Let $\varphi:\R\to \R$ such that the multiplier with symbol $m(\xi)=\varphi(\log |\xi|)$ is bounded on $L_p(\R^d)$. Then for almost every $s,t\in \R$,
  \[ |\varphi(s) - \varphi(t)| \lesssim \|T_m\|_p \left(\int_0^{|s-t|^{1/4}} \frac{\fonction_d(\delta)^{|\frac 1 p - \frac 1 2|}}{\delta} d\delta + |s-t|^{\frac 1 4}\right).\]
\end{cor}
\begin{proof} Without loss of generality we can assume $\|T_m\|=1$. Write
  \[ c_n = \inf_{\delta \in (0,1)} \left(\fonction_d(\delta)^{\big|\frac 1 p - \frac 1 2\big|} + \frac{1}{2^n \delta^2} \right).\]
  Taking $\delta \in [2^{-(n+1)/3},2^{-n/3}]$ and integrating, we have
  \[ c_n \leq \frac{3}{\log 2} \int_{2^{-(n+1)/3}}^{2^{-n/3}} \frac{\fonction_d(\delta)^{\big|\frac 1 p - \frac 1 2\big|}}{\delta}d\delta + 2^{-(n-2)/3},\]
  and therefore for every integer $N$,
  \[ \sum_{n \geq N} c_n \lesssim \int_0^{2^{-N/3}} \frac{\fonction_d(\delta)^{\big|\frac 1 p - \frac 1 2\big|}}{\delta}d\delta + 2^{-N/3}.\]
Let $s,t \in \R$ with $|s-t|\leq 1$. Lemma~\ref{lem:continuity_from_LittlewoodPaley} and Theorem~\ref{thm:Lpbddness_LP_quantitative} allows us to write for almost every $s,t$,
  \begin{align*} |\varphi(s)-\varphi(t)| & \lesssim_d  \sum_{n} \min(1, 2^n|s-t|) c_n\\
    & \leq \sum_{n > N} c_n + |s-t| \sum_{n\leq N}  2^n c_n,
  \end{align*}
  for any integer $N \geq 0$. The sequence $2^n c_n$ is clearly non-decreasing, so $\sum_{n \leq N} 2^n c_n \leq (N+1) 2^N c_N$. So if $N \geq 0$ is the largest integer such that $|s-t| (N+1) 2^N \leq 1$, then we have $2^N=O(\frac{1}{|s-t| |\log |s-t||})$, so for $|s-t|$ small enough $2^N \leq |s-t|^{-3/4}$, and 
\[ |\varphi(s)-\varphi(t)| \lesssim \sum_{n \geq N} c_N \lesssim \int_0^{|s-t|^{1/4}}  \frac{\fonction_d(\delta)^{\big|\frac 1 p - \frac 1 2\big|}}{\delta}d\delta + 2^{-N/4}.\qedhere\]\end{proof} 
\subsection{Relaxation on the Besicovich construction}
We end this short discussion with a comment. In the definition of $\fonction_d(\delta)$, we can replace the requirement that the $\overline{R_i}$ are pairwise disjoint by the requirement that they are almost pairwise disjoint. For example, for $4/3\leq p\leq 4$ and $1<a<b\leq \infty$, Theorem~\ref{thm:Lpbddness_LP_quantitative} remains true of instead of $\fonction_d(\delta)$ we consider the smallest $\varepsilon$ such that there is a family of $\delta$-tubes such that 
\begin{equation}\label{eq:gooddeltatubesbis}   \frac{\Big| \bigcup_j R_j \Big|}{\sum_j |R_j|} \cdot \frac{\sum_{i,j} |\overline{R_i}^{[a,b]}\cap \overline{R_j}^{[a,b]}|}{\sum_j |R_j|} \leq \varepsilon.
\end{equation}
Obviously for $[a,b]=[2,3]$ this is smaller that $\fonction_d(\delta)$ because $\frac{\sum_{i,j} |\overline{R_i}\cap \overline{R_j}|}{\sum_j |R_j|}=1$ when the $\overline{R_i}$ are pairwise disjoint. The only difference in the proof is the following variant of Lemma~\ref{lem:consequence_Kakeya}.
  \begin{lem} Let $\frac 4 3  \leq p \leq 2$ and $q=\frac{p}{p-1}$ the conjugate exponent. Let $R_j$ be a family of $\delta$-tubes satisfying \eqref{eq:gooddeltatubesbis}. Let $f,g \in C^\infty_c(\R)$ and $\sigma,\rho \in C^\infty_c(\R^{d-1})$ supported in $[1,2]$, $[a,b]$, $B(0,1)$ and $B(0,1)$ respectively. For every tube $R_j$ image of $[0,1] \times B(0,\delta)$ by the isometry $U_j$, let $f_j$ and $g_j$ be the images of the functions $(x,y)\mapsto f(x) \delta^{-\frac{d-1}p} \sigma(y/\delta)$ and $(x,y)\mapsto g(x) \delta^{-\frac{d-1}q} \rho(y/\delta)$ by $U_j$. Then
    \begin{equation}\label{eq:LpLqsquareinequalitygain} \Big\|(\sum_{j=1}^N |f_j|^2)^{\frac 1 2}\Big\|_p \Big\| (\sum_{j=1}^N |g_j|^2)^{\frac 1 2}\Big\|_q \leq C N \varepsilon^{\frac 1 p - \frac 1 2},
\end{equation}
    for a real $C=C(f,g,\rho,\sigma,p)$.
    %=c^{\frac 2 p -1} N \|f\|_2 \|g\|_{\frac{2p}{3p-4}} \|\rho\|_2\|\rho\|_{\frac{2q}{3p-4}}.\]    where
    %%     \begin{equation}\label{eq:LpLqsquareinequalitygain} \Big\|(\sum_{j=1}^N |f_j|^2)^{\frac 1 2}\Big\|_p \Big\| (\sum_{j=1}^N |g_j|^2)^{\frac 1 2}\Big\|_q \leq \sqrt{c}\|f\|_2 \|g\|_\infty \|\rho\|_{\infty} N \varepsilon^{\frac 1 p - \frac 1 2}.
%% \end{equation}
  \end{lem}
  \begin{proof}
    The case $p=2$ is clear with $C(f,g,\rho,2)= \|f\|_2 \|g\|_2 \|\sigma\|_2 \|\rho\|_2 $. We will prove that for $p =\frac 4 3$, \eqref{eq:LpLqsquareinequalitygain} holds with $C(f,g,\sigma,\rho,4/3)=\sqrt{c} \|f\|_2 \|g\|_\infty \|\sigma\|_2 \|\rho\|_\infty$, where $c$ is the volume of the unit ball in $\R^{d-1}$. By interpolation, this will imply the general case with
    \[C(f,g,\sigma,\rho,p)=c^{\frac 2 p -1} \|f\|_2 \|g\|_{\frac{2p}{3p-4}} \|\sigma\|_2\|\rho\|_{\frac{2q}{3p-4}}.\]
    So let us assume $p =\frac 4 3$, so $q=4$. 
    Observe that $f_j$ and $g_j$ are supported in $R_j$ and $\overline{R_j}^{[a,b]}$ respectively. By H\"older's inequality we bound
    %    \[ \Big\|(\sum_j |f_j|^2)^{\frac 1 2}\Big\|_{4/3} \leq \left(c \frac{|\bigcup_{j=1}^N R_j|}{\sum_j |R_j|}\right)^{\frac 1 4} N^{\frac 3 4} \|f\|_2 \|\rho\|_2 .\]
    \[ \Big\|(\sum_j |f_j|^2)^{\frac 1 2}\Big\|_{\frac 4 3} \leq |\bigcup_{j=1}^N R_j|^{\frac 1 4} \Big\|(\sum_j |f_j|^2)^{\frac 1 2}\Big\|_{2} =  |\bigcup_{j=1}^N R_j|^{\frac 1 4} N^{\frac 1 2} \delta^{-\frac{d-1}{4}} \|f\|_2 \|\sigma\|_2 .\]
Consider the functions $N_\mathcal{R},\overline{N}_{\mathcal R}:\R^d\to \N$
\begin{align*} N_{\mathcal R}(x) = \sum_{j=1}^N 1_{x \in R_j}, \ \ \ 
  \overline{N}_{\mathcal R}(x) = \sum_{j=1}^N 1_{x \in \overline{R_j}^{[a,b]}}.
\end{align*}
%% \begin{align*} N_{\mathcal R}(x) &= \sum_{j=1}^N 1_{x \in R_j},\\
%%   \overline{N}_{\mathcal R}(x) &= \sum_{j=1}^N 1_{x \in \overline{R_j}}.
%% \end{align*}
Then we have
\[ |\bigcup_{j=1}^N R_j| = \sum_j \int_{R_j} \frac{1}{N_{\mathcal{R}}}(x) dx, \ \ \sum_{i,j} |\overline{R_i}^{[a,b]} \cap \overline{R_j}^{[a,b]}| =\|\overline{N}_{\mathcal R}\|_{2}^2.\]
Moreover we have the pointwise inequality
\[ (\sum_j |g_j|^2)^{\frac 1 2} \leq \|g_j\|_\infty \overline{N}_{\mathcal R}^{\frac 1 2} = \|g\|_\infty \|\rho\|_\infty \delta^{-\frac{d-1}{4}} \overline{N}_{\mathcal R}^{\frac 1 2},\]
so
\begin{align*} \Big\| (\sum_{j=1}^N |g_j|^2)^{\frac 1 2}\Big\|_4 &\leq  \|g\|_\infty \|\rho\|_\infty \delta^{-\frac{d-1}{4}} \|\overline{N}_{\mathcal R}\|_{2}^{\frac 1 2}\\
  &= \|g\|_\infty \|\rho\|_\infty \delta^{-\frac{d-1}{4}} \big(\sum_{i,j} |\overline{R_i} \cap \overline{R_j}|\big)^{\frac 1 4}.
%  & \leq C \|\overline{N}_{\mathcal R}\|_{L_{1}}^{\frac 2 q - \frac 1 2}} \|\overline{N}_{\mathcal R}\|_{L_{2}}^{1 - \frac 2 q}},
\end{align*}
%% But
%% \[  = N c \delta^{d-1}  \frac{\sum_{i,j} |\overline{R_i} \cap \overline{R_j}|}{\sum_i |R_i|}.\]
%% Therefore, we obtain
%% \[ \Big\| (\sum_{j=1}^N |g_j|^2)^{\frac 1 2}\Big\|_4 \leq .\]
To conclude the proof, it remains to observe that $\delta^{-\frac{d-1}{2}} N^{\frac 1 2} = \frac{\sqrt{c} N}{(\sum_i |R_i|)^{\frac 1 2}}$, so
\[ \delta^{-\frac{d-1}{2}} N^{\frac 1 2} \big(\sum_{i,j} |\overline{R_i} \cap \overline{R_j}|\big)^{\frac 1 4} |\bigcup_{j=1}^N R_j|^{\frac 1 4} \leq \sqrt{c} N \varepsilon^{\frac 1 4}.\qedhere\]
\end{proof}

\section{Schur multipliers on the sphere}\label{sec:Schur}
Theorem~\ref{thm:Spbddness_LP_quantitative} will easily follow from the following lemma.
\begin{lem}\label{lem:Spbddness_quantitative} There is $h \in C^\infty_c(\R)\setminus\{0\}$ and a constant $C_d$ such that, for every bounded measurable function $m:(-1,1)\to \C$ such that the Schur multiplier with symbol $(\xi,\eta) \in \sphere^d \times \sphere^d \mapsto m(\langle \xi,\eta\rangle)$ is $S_p$-bounded with norm $1$ and every $r \in \R^*$ and $\delta \in (0,1)$,
  \[ \big| \int_\R m(\cos(\theta+t/r)) \widehat h(t) dt\big| \leq C_d \big( \fonction_d(\delta)^{\big|\frac 1 p - \frac 1 2\big|} + \frac{1}{|r \delta^2 \sin \theta|^{\frac 1 3} }\big).\]
\end{lem}
\begin{proof}[Proof of Theorem~\ref{thm:Spbddness_LP_quantitative} using Lemma~\ref{lem:Spbddness_quantitative}] For $n\geq 1$, the first inequality is Proposition~\ref{prop:from_one_h_to_all} and Example~\ref{prop:from_one_h_to_all}, while the second follows from Lemma~\ref{lem:infi_for_particular_fd} and Keich's estimates \eqref{eq:Keich}. For $n=0$, the theorem is clear because the left-hand side is $\lesssim \|m\|_{L_\infty}$ and the right-hand side is larger.
\end{proof}

\subsection{Proof of Lemma~\ref{lem:Spbddness_quantitative}}
The idea is to adapt the arguments in the previous section to spherical geometry and Schur multipliers. To stick back directly to the Euclidean setting we will use the stereographic projection. An important computation is to analyse how two stereographic projections through two distinct poles are distorted (Lemma~\ref{eq:Taylor_expansion_in_r}). To relate Fourier multipliers with Euclidean analysis, we will also exploit the ideas of Schur-Fourier transference \cite{MR2866074,CaspersdlS} (Lemma~\ref{lem:square_inequality_Schur}) together with the amenability of $\R^d$. Then in Lemma~\ref{lem:change_of_variable_Schur} we prove a replacement of the change of variables argument from Lemma~\ref{lem:technical}.

There are three small scales at which we will work. The first is the scale $1/r$, that we think as the smallest, for the parameter $r$ that is in the statement of Lemma~\ref{lem:Spbddness_quantitative}. The third is the scale $\delta$, that we think of much larger than $1/r$. The second (intermediate) scale will be a scale $1/R$, for a free parameter $R$. There is a tension between choosing $R$ large so that the balls of radius $R$ in $\R^d$ are good F\o{}lner sequences and therefore Schur multipliers at the scale $1/R$ behave as Fourier multipliers, and choosing $R$ small compared to $r$ so that the stereographic distortion from Lemma~\ref{eq:Taylor_expansion_in_r} is negligeable. We will see that the best choice will be $R=(r \sin \theta/\delta)^{\frac 1 3}$.

Let $\rho \in C^\infty_c(\R^{d-1})$ such that $\|\rho\|_{L_2} = 1$ and $f,g \in C^\infty_c(\R)$ nonzero supported in $(-1,0)$ and $(1,2)$ respectively. We will prove Lemma~\ref{lem:Spbddness_quantitative} with $h = f \ast g^*$, where $g(t) = \overline{g(-t)}$.

Let us fix $d,p,m$ as in Lemma~\ref{lem:Spbddness_quantitative}. By duality we can assume that $p<2$. The lemma is obvious for $|r|\leq 1$, so we can assume $r\geq 1$. By symmetries we can also assume $\theta \in (0,\pi)$. For simplicity we also assume $r>1$; the case $r<-1$ is identical, replacing all bounds $\frac{1}{r \sin\theta}$ by $\frac{1}{|r \sin\theta|}$.

Consider the function $\widetilde{m_r}:\R^d\times \R^d \to \C$ defined by
\begin{equation}\label{eq:defmtilde} \widetilde{m_r}(x,y) = m\Big(\Big\langle \frac{e_{d+1}+ x/r}{|e_{d+1}+ x/r|},V_\theta \frac{e_{d+1}- y/r}{|e_{d+1}- y/r|}\Big\rangle \Big),
\end{equation}
where $V_\theta \in O(d+1)$ is the reflection matrix
\begin{equation} V_\theta=\begin{pmatrix} 1_{d-1} & 0 & 0\\ 0& -\cos \theta & \sin \theta \\ 0 & \sin \theta & \cos \theta 
  \end{pmatrix}.
\end{equation}

We have the following:
\newcommand\function{\psi_r}
\begin{lem}\label{eq:Taylor_expansion_in_r} The function $\function:\R^d \times\R^d\to \R$ given by
  \[\function(x,y) = r\Big(\arccos\Big\langle \frac{e_{d+1}+x/r}{|e_{d+1}+ x/r|},V_\theta \frac{e_{d+1}-y/r}{|e_{d+1}-y/r|}\Big\rangle-\theta\Big)\] is smooth and satisfies
  \[|\function(x,y) - (x_d-y_d)| \lesssim \frac{|x|^2+|y|^2}{r \sin \theta},\]
  \[\Big|\frac{\partial \function}{\partial{x_d}}(x,y) - 1\Big| \lesssim \frac{|x|+|y|}{r \sin \theta}.\]
  The implicit constants do not depend on $r$ and $\theta$.
\end{lem}
\begin{proof} Consider the function on $(\R^{d+1}\setminus\{0\}) \times (\R^{d+1}\setminus\{0\})$
  \[ \beta(X,Y) = \arccos \Big\langle \frac{X}{|X|},\frac{Y}{|Y|}\Big\rangle.\]
    A small computation shows that $\beta$ is smooth outside of $\{(X,Y) \mid \R X=\R Y\}$ and that
    \begin{align*} \nabla_X\beta(X,Y) &= - \frac{1}{|X|} \frac{P_{X^\perp}(Y)}{|P_{X^\perp}(Y)|}.
    \end{align*}
    %% Indeed, assuming without loss of generality that $|X|=|Y|=1$, we have $\frac{X+H}{|X+H|} = P_{X^\perp}(H)+O(|H|^2)$ and
    %% \[ \langle \frac{X}{|X+H|},Y\rangle = \langle X,Y\rangle + \langle H,P_{X^\perp}(Y)\rangle.\]
    %% Using that $\arccos'(\langle X,Y\rangle) = \frac{-1}{\sqrt{1-\langle X,Y\rangle^2}} = \frac{-1}{|P_{X^\perp}(Y)|}$, we obtain
    %% \begin{align*} \beta(X+H,Y) &= \arccos(\langle X,Y\rangle + \langle H, P_{X^\perp}Y\rangle + O(|H|^2))\\
    %%   & = \beta(X,Y) -\frac{1}{|P_{X^\perp}(Y)|}\langle H,P_{X^\perp}(Y)\rangle + O(|H|^2).
    %% \end{align*}
    Here $P_{X^\perp}$ denotes the orthogonal projection on the orthogonal subspace $X^\perp \subset \R^{d+1}$, and $\nabla_X$ is the gradient in the first $d+1$ (out of $2d+2$) variables. From this formula, we see that for any pair $X_0,Y_0 \in \R^{d+1}$  of unit vectors and any pair $X,Y \in \R^{d+1}$ of vectors of norm $\geq 1$
    \[ |\nabla_X \beta(X,Y) - \nabla_X \beta(X_0,Y_0)| \lesssim \frac{|X-X_0| + |Y-Y_0|}{|P_{X_0^\perp}(Y_0)|}.\]
    If we specify to $X_0 = e_{d+1},Y_0=V_\theta e_{d+1}$, we have $\beta(X_0,Y_0) = \theta$, $|P_{X_0^\perp}(Y_0)|=\sin \theta$, $\nabla_X \beta(X_0,Y_0) = e_d$ and we deduce
    \begin{equation}\label{eq:estimate_gradbeta} |\nabla_X\beta (X,Y) - e_d|\lesssim \frac{|X-X_0| + |Y-Y_0|}{\sin \theta}.
    \end{equation}
By our definition, we have $\function(x,y) = r (\beta(e_{d+1}+x/r,V_\theta(e_{d+1}-y/r)) - \beta(X_0,Y_0))$, so by \eqref{eq:estimate_gradbeta} we obtain
    \[ | \nabla_x \function(x,y) - e_d| \lesssim \frac{|x| + |y|}{r\sin \theta}.\]
This in particular implies the second inequality in the lemma. For the first, use that $V_\theta=V_\theta^*$ to deduce that $\function(x,y)=\function(-y,-x)$ and 
\[ | \nabla_y \function(x,y) + e_d| \lesssim \frac{|x| + |y|}{r\sin \theta}.\]
Therefore, the first inequality follows by integrating $\function$ on the interval $[0,(x,y)]$.
  \end{proof}

As in \S~\ref{sec:quantitative}, let $\{R_j\mid 1 \leq j \leq N\}$ be a family of $\delta$-tubes in $\R^d$ such that the $\overline{R_j}$ are pairwise disjoint and that satisfy \eqref{eq:gooddeltatubes} for some $\varepsilon>0$, let $\vecteur_j=\vecteur(R_j)$ and define the functions $f_j,g_j$ as in \eqref{eq:def_of_fj} and \eqref{eq:def_of_gj}. Choose also $U_j \in O(d)$ such that $U_j \vecteur_j =e_d$, and define $S_j$ the Schur multiplier with symbol $(x,y)\mapsto \widetilde{m}_r(U_j x,U_j y)$. 

Fix a radial function $\varphi_1 \in \C^\infty_c(\R^d)$ with nonnegative values, with $L_2(\R^d)$-norm $1$ and supported in the ball of radius $1$. For $R>0$, set $\varphi_R(\xi)=\varphi_1(\xi/R)R^{-\frac d 2}$. Then $\varphi_R$ is radial, is supported in the ball of radius $R$, has $L_2$-norm $1$. 
Define an operator $A_j \in B(L_2(\R^d))$ by its kernel
\[ A_j = (\varphi_R(x)^{\frac 1 p} \widehat{f_j}(x-y) \varphi_R(y)^{\frac 1 p})_{x,y \in \R^d}\]
and similarly
\[ B_j = (\varphi_R(x)^{\frac 1 q} \widehat{g_j}(x-y) \varphi_R(y)^{\frac 1 q})_{x,y \in \R^d}.\]
\begin{lem}\label{lem:square_inequality_Schur} If $f_1,\dots,f_N \in L_p(\R^d)$ and $g_1,\dots,g_N \in L_p(\R^d)$ are defined as above,
  \begin{align*} |\sum_j \Tr( S_j(A_j) B_j^*)| &\leq \|(\sum_j |{f}_j|^2)^{\frac 1 2}\|_{L_p} \|(\sum_j |{g}_j|^2)^{\frac 1 2}\|_{L_q} \lesssim N \varepsilon^{\frac 1 p - \frac 1 2}.
  \end{align*}
\end{lem}
\begin{proof} The second inequality is Lemma~\ref{lem:consequence_Kakeya}, so we only have to prove the first. Write $A = \sum_j A_j \otimes e_{1,j} \in B(L_2(\R^d))\otimes M_N(\C) \simeq B(L_2((\R^d)^N))$ and similarly for $B$. The map $f_j\mapsto A_j$ extends to a completely contractive map $L_p(\R^d)\to S_p(L_2(\R^d))$. This is for example proved in \cite[Theorem 5.1]{CaspersdlS} in the case where $\varphi_R$ a normalized indicator function, but the proof is identical.\footnote{We recall the proof. By interpolation, it is enough to treat the case $p=\infty$ and $p=1$. The case $p=\infty$ is obvious, because by Fourier transform the operator with kernel $\widehat{f_j}(x-y)$ corresponds to the operator of multiplication by $f_j$ on $L_2(\R^d)$. The case $p=1$ is easy: if $F \in L_1(\R^d;S_1)$, it can be written as $g^* h$ for $g,h\in L_2(\R^d;S_2)$ satisfying $\|g\|_2\|h\|_2 = \|f\|_1$. Then $A_j$ factors as $Y^* X$ when $X=(\widehat{h}(x-y) \varphi_R(y))_{x,y \in \R^d}$ and similarly for $Y$ and $g$. By Fubini and Plancherel, we can compute $\|X\|_{S_2} = \|h\|_2 \|\varphi_R\|_2 = \|h\|_2$ and $\|Y\|_{S_2} = \|g\|_2$, so $\|A_j\|_{S_1} \leq \|X\|_2 \|Y\|_2 = \|f_j\|_1$.} In particular we have
\begin{equation}\label{eq:square_pinequality_transference} \|A\|_p \leq \| \sum_j f_j \otimes e_{1,j}\|_{L_p(\R^d;S_p^N)} =  \|(\sum_j |{f}_j|^2)^{\frac 1 2}\|_{L_p}
\end{equation}
and similarly
\begin{equation}\label{eq:square_qinequality_transference} \|B\|_q\leq \|(\sum_j |{g}_j|^2)^{\frac 1 2}\|_{L_q}.
\end{equation}

The linear map $X = (X_{i,j})_{i,j \leq N}\mapsto (S_j(X_{i,j}))$ is a Schur multiplier with symbol $S:(x,i),(y,j)\mapsto \widetilde{m}_r(U_j x,U_jy)$. By the definition of $\widetilde{m}_r$, we can rewrite
  \[ S((x,i),(y,j)) = m\Big(\Big\langle \frac{e_{d+1}+x/r}{|e_{d+1}+x/r|},\begin{pmatrix} U_j^*&0\\0&1
  \end{pmatrix}  V_\theta \begin{pmatrix} U_j&0\\0&1
  \end{pmatrix}
  \frac{e_{d+1}-y/r}{|e_{d+1}-y/r|}\Big\rangle \Big).\]
  From this expression, we see from Lemma~\ref{lem:restriction} that $S$ is bounded on $S_p$ with norm $\leq 1$ (recall that we have assumed that $(\xi,\eta) \in \sphere^d \times \sphere^d \mapsto m(\langle \xi,\eta\rangle)$ is $S_p$-bounded with norm $1$). As a consequence, we see that
  \begin{align*} |\sum_j \Tr( S_j(A_j) B_j^*)|& = | \Tr( S(A)B^*)| \leq \|A\|_p \|B\|_p.
  \end{align*}
  The lemma follows from \eqref{eq:square_pinequality_transference} and \eqref{eq:square_qinequality_transference}.
  \end{proof}

\begin{lem}\label{lem:change_of_variable_Schur} There is a constant $C=C(d,\rho)$ such that if we set $h=f\ast g^*$, 
  \[ \Big| \Tr( S_j(A_j) B_j^*) - \int_\R m(\cos(\theta + \frac t r))  \widehat h(t) dt \Big|\leq C\Big(\frac{1}{\delta R}+ \frac{R^2}{r\sin\theta}\Big)\|(1+|s|)\widehat{h}\|_1.\]
  %  \[ \Big| \Tr( S_j(A_j) B_j^*) - \int_\R m(\cos(\theta + \frac t r))  \widehat h(t) dt \Big|\leq C\Big(\frac{1}{\delta R}\|(1+|s|)\widehat{h}\|_1  + \frac{R^2}{r\sin \theta}\|\widehat{h}'\|_1\Big).\]
\end{lem}
\begin{proof}
  Developping we have
  \[  \Tr( S_j(A_j) B_j^*) = \iint \widetilde{m}_r(U_jx,U_jy) \varphi_R(x) \varphi_R(y) (\widehat{f_j} \overline{\widehat{g_j}})(x-y) dx dy.\]
  We can make the change of variable $x\mapsto U_jx$, $y\mapsto U_jy$ and observe that ($\varphi_R$ is radial) this quantity is equal to
  \[ \iint \widetilde{m}_r(x,y) \varphi_R(x) \varphi_R(y) (\widehat{f_j} \overline{\widehat{g_j}})(U_j^{-1}(x-y)) dx dy.\]
  Moreover, as in \eqref{eq:FT_of_f_jg_j}, in the coordinates $x=(x',x_d) \in \R^{d-1}\times\R$, it becomes
  \[ \iint \widetilde{m}_r(x,y) \varphi_R(x) \varphi_R(y) \widehat{h}(x_d-y_d) \delta^{d-1} |\widehat\rho(\delta(x'-y'))|^2 dx dy.\]
  Make a change of variable $\omega= \delta(x'-y') \in \R^{d-1}, s=x_d-y_d \in \R$, $\xi=y/R \in \R^d$ in the preceding equation. If we write $F = |\widehat{\rho}|^2$ and $M_r(\xi,\omega,s) = \widetilde{m}_r(R\xi+(\frac{\omega}{\delta},s),R\xi)$, we get
  \[ \Tr( S_j(A_j) B_j^*) = \iint M_r(\xi,\omega,s) \varphi_1(\xi+(\frac{\omega}{\delta R},\frac{s}{R})) \varphi_1(\xi) \widehat{h}(s) F(\omega) ds d\omega d\xi.\]
  From this expression, we readily see that the left-hand side in the statement of the lemma is $\leq 2 \|\widehat{h}\|_1$, so the lemma is obvious with constant $2/c$ if $\frac{1}{\delta R}+ \frac{R^2}{r\sin\theta} \geq c$.
  
Since $\varphi_1$ is $C^1$ and compactly supported, we have
  \[\varphi_1(\xi + (\frac{\omega}{\delta R},\frac{s}{R})) =\varphi_1(\xi) + O\Big(\frac{|\omega|}{R\delta} + \frac{s}{R}\Big),\]
  and we deduce
  \begin{align}\label{eq:value_of_TrTmAB}  \Tr( S_j(A_j) B_j^*) = \iint M_r(\xi,\omega,s) \varphi_1(\xi)^2 \widehat{h}(s) F(\omega) ds d\omega d\xi\\ + O\Big(\frac{1}{\delta R} \|(1+|w|) F\|_1 \|(1+|s|)\widehat{h}\|_1\Big).\notag
  \end{align}
  Let $\Omega=\{(\xi,\omega,s) \mid |\xi|\leq 1, |\omega|\leq  \delta R, |s|\leq R\}$. Then (recall $\varphi_1$ is supported in the ball of radius $1$) we have
  \[ \|\widehat h\|_{L_1([-R,R]^c)} + \|F\|_{L_1(\R^{d-1} \setminus B(0, \delta R))} \leq \frac{1}{\delta R} \|(1+|w|) F\|_1 \|(1+|s|)\widehat{h}\|_1,\]
  and we obtain
  \begin{align}\label{eq:value_of_TrTmAB_after_truncation}  \Tr( S_j(A_j) B_j^*) = \iint_\Omega  M_r(\xi,\omega,s) \varphi_1(\xi)^2 \widehat{h}(s) F(\omega) ds d\omega d\xi\\ + O\Big(\frac{1}{\delta R} \|(1+|w|) F\|_1 \|(1+|s|)\widehat{h}_1\Big).\notag
  \end{align}
  
  For fixed $\xi,\omega$, we make the change of variable $s \mapsto t=\function( R\xi + (\frac \omega \delta,s), R\xi)$ from Lemma~\ref{eq:Taylor_expansion_in_r}. To do so, we can assume that $R \leq  \alpha r \sin \theta $ for a constant $\alpha>0$ so that the bound on $\partial \function/\partial x_d-1$ is $\leq \frac 1 2$, and that the image $|s|\leq R\mapsto t$ contains $[-R/2,R/2]$. Indeed, as we have already argued, otherwise the lemma is obvious. In that case, we see that the change of variable is bijective and we denote by $t\mapsto s_{\xi,\omega}(t)=s$ the inverse map. If $\Sigma$ is the image of $\Omega$ by $(\xi,\omega,s)\mapsto(\xi,\omega,t)$, then $\Sigma$ contains $B(0,1) \times B(0,\delta R) \times [-R/2,R/2]$. Moreover, by definition of $M_r$, $M_r(\xi,\omega,s) = m (\cos(\theta + \frac{t}{r}))$, so by Lemma~\ref{eq:Taylor_expansion_in_r} the integral in \eqref{eq:value_of_TrTmAB_after_truncation} is
\[ \iint_{\Sigma} \varphi_1(\xi)^2  m(\cos(\theta+\frac t r)) \widehat h(s)F(\omega) dt d\omega d\xi + O\Big(\frac{R}{r \sin \theta} \|\widehat{h}\|_1\Big).\]
%%   Define $\Sigma_1 = \{(\xi,\omega,t) \in \Sigma \mid R|\xi|+|\omega|/\delta \leq |s|\}$ and $\Sigma_2=\Sigma \setminus\Sigma_1$. On $\Sigma_1$, we have
%%   \[ |t - s| =  O(\frac{s^2}{r \sin \theta}) = O(\alpha s),\]
%%   so if $\alpha$ is chosen small enough we have
%%   \[ |t-s| \leq \frac{1}{2}|s|.\]
%% This fixes our choice of $\alpha$. The preceding implies that $|s| \leq 2|v|$ for every $v$ in the interval between $t$ and $s$. Therefore
%%   \[ |t-v| \leq |t-s| = O(\frac{s^2}{r \sin \theta}) = O(\frac{v^2}{r\sin \theta}).\]
%%   On the other hand, on $\Sigma_2$ we have
%%   \[|t-v| =  O\Big(\frac{R^2|\xi|^2+|\omega|^2/\delta^2}{r\sin \theta}\Big) = O\Big(\frac{R^2}{r\sin \theta}\Big).\]
%%     To summarize, there is a constant $C$ such that on $\Sigma$ we have for every $v \in [t,s]$
%%     \[|t-v|\leq \frac{C}{r \sin \theta} \big(R^2+|v|^2\big).\]
On $\Sigma$, Lemma~\ref{eq:Taylor_expansion_in_r} implies
\[ |t-s| \leq C\frac{R^2}{r\sin \theta},\]
for a constant $C$. Therefore, 
  \[|\widehat{h}(t) - \widehat{h}(s)| \leq \int_\R |\widehat{h}'(v)| 1_{|t-v| \leq  \frac{CR^2}{r \sin \theta}}dv.\]
Integrating, we obtain
    \begin{multline*} \Big| \iint_{\Sigma} \varphi_1(\xi)^2  m(\cos(\theta+\frac t r)) (\widehat h(s) - \widehat{h}(t)) F(\omega) dt d\omega d\xi \Big| \\
      \leq  \iiint_{\R^{2d}} \varphi_1(\xi)^2 |\widehat h'(v)| F(\omega) \frac{2CR^2}{r\sin \theta} dv d\omega d\xi = \frac{2C R^2}{r\sin \theta}\|\widehat{h}'\|_1.
    \end{multline*}
Observe that, since $h$ is supported in $[-3,-1]$, there is a universal constant such that $\|\widehat{h}'\|_1 \lesssim  \|\widehat h\|_1$. This proves that the integral in \eqref{eq:value_of_TrTmAB_after_truncation} is
    \[ \iint_{\Sigma} \varphi_1(\xi)^2  m(\cos(\theta+\frac t r)) \widehat{h}(t) F(\omega) dt d\omega d\xi + O\Big(\frac{R^2}{r \sin \theta} \|\widehat{h}\|_1 \Big).\]
    Finally, since $\Sigma$ contains $\R^d \times B(0,\delta R) \times [-R/2,R/2]$, this integral is
\begin{multline*} \iint_{\R^d\times \R^d} \varphi_1(\xi)^2 m(\cos(\theta + \frac t r)) \widehat h(t) F(\omega) d\xi dt d\omega\\ + O(\|\widehat h\|_{L_1([-R/2,R/2]^c)} + \|F\|_{L_1(\R^{d-1} \setminus B(0, \delta R))}) = O\Big(\frac{1}{\delta R} \|(1+|w|) F\|_1 \|(1+|s|)\widehat{h}\|_1\Big).
\end{multline*}
The last integral can be simplified to $\int m(\cos(\theta + \frac t r))$, because $\|\varphi_1\|_2^2 = \int F = 1$. Putting everything together, we get
\begin{multline*} \Big| \Tr( S_j(A) B^*) - \int_\R m(\cos(\theta + \frac t r)) \widehat h(t) dt \Big| \\\lesssim \frac{1}{\delta R}\|(1+|w|) F\|_1 \|(1+|s|)\widehat{h}\|_1  + \frac{R^2}{r\sin \theta}\|\widehat{h}\|_1.\qedhere
\end{multline*}
\end{proof}

We can now conclude the proof of Lemma~\ref{lem:Spbddness_quantitative}. By Lemma~\ref{lem:square_inequality_Schur} and Lemma~\ref{lem:change_of_variable_Schur}, we get
\[ \Big|\int_\R m(\cos(\theta+t/r)) \widehat h(t) dt\Big| \lesssim \frac{1}{\delta R} + \frac{R^2}{r\sin \theta} + \varepsilon^{\frac 1 p - \frac 1 2},\]
with implicit constants depending in $d,f,g,\rho$. This is true for every $R$ and every family of $\delta$-tubes, so taking $R = \Big(\frac{r\sin \theta}\delta\Big)^{\frac 1 3}$ and taking the infimum over all such tubes, we obtain
\[ \Big|\int_\R m(\cos(\theta+t/r)) \widehat h(t) dt\Big| \lesssim \frac{1}{\big(r\delta^2 \sin \theta \big)^{\frac 1 3}} + \fonction_d(\delta)^{\frac 1 p - \frac 1 2}.\]

\subsection{Consequence}
Let us state the following consequence. Again, the exponent $\frac 1 3$ is not optimized.
\begin{cor}\label{cor:Spbddness_LP_quantitative_precise}
Let $m:(-1,1) \to \C$ such that the Schur multiplier with symbol $(\xi,\eta)\mapsto m(\langle \xi,\eta\rangle)$ is bounded on $S_p(L_2(\R^d))$ with norm $1$. Then for almost every $s, t\in \R$,
\[ |m(\cos s) - m(\cos t)| \lesssim  \int_0^{|s-t|^{1/3}} \frac{\fonction_d(\delta)^{|\frac 1 p - \frac 1 2|}}{\delta} d\delta + |s-t|^{\frac 1 9} \big(\frac{1}{|\sin s|^{\frac 1 3}}+\frac{1}{|\sin s|^{\frac 1 3}}\big)\]
with implicit constant depending on $d$.
\end{cor}
\begin{proof} Write
  \[ c_n(\theta) = \inf_{\delta \in (0,1)} \left(\fonction_d(\delta)^{\big|\frac 1 p - \frac 1 2\big|} + \frac{1}{|2^n \delta^2 \sin \theta|^{1/3}} \right).\]
  Taking $\delta \in [2^{-(n+1)/3},2^{-n/3}]$ and integrating, we have
  \[ c_n(\theta) \leq \frac{3}{\log 2} \int_{2^{-(n+1)/3}}^{2^{-n/3}} \frac{\fonction_d(\delta)^{\big|\frac 1 p - \frac 1 2\big|}}{\delta}d\delta + \frac{1}{|2^{(n-2)/3} \sin \theta|^{\frac 1 3}},\]
  and therefore for every integer $N$,
  \begin{equation}\label{eq:bound_cns} \sum_{n \geq N} c_n(\theta) \lesssim \int_0^{2^{-N/3}} \frac{\fonction_d(\delta)^{\big|\frac 1 p - \frac 1 2\big|}}{\delta}d\delta + \frac{1}{|2^{N/3} \sin \theta|^{\frac 1 3}}.
  \end{equation}

Write $\varphi=m\circ \cos$ and $\varphi_n = W_n \ast \varphi$. As in the proof of Lemma~\ref{lem:continuity_from_LittlewoodPaley}, Theorem~\ref{thm:Spbddness_LP_quantitative} allows us to write for almost every $s<t$ and every $N\geq 0$,
\begin{align*} |\varphi(s)-\varphi(t)| & \leq  \sum_{n\leq N} \int_{s}^t |\varphi_n'(\theta)| d\theta + \sum_{n>N} |\varphi_n(s)|+|\varphi_n(t)|\\
  & \leq \sum_{n \leq N} \int_{s}^t 2^n c_n(\theta) d\theta + \sum_{n>N} c_n(s)+c_n(t)\\
  & \leq \sum_{n \leq N} |s-t| 2^n c_n(s) + \sum_{n>N} c_n(s)+c_n(t).
  \end{align*}
The inequality $|\varphi_n'(\theta)| = |(W_n' \ast \varphi)(\theta)| \lesssim 2^n c_n(\theta)$ is where we use Theorem~\ref{thm:Spbddness_LP_quantitative}, using also Proposition~\ref{prop:from_one_h_to_all} and Example~\ref{ex:HardyLittlewood}, and the inequality $\int_{s}^t c_n(\theta) d\theta \leq |s-t| c_n(s)$ is Example~\ref{ex:HardyLittlewood}.

The sequence $2^{n/3} c_n(s)$ is clearly non-decreasing, so
  \[ \sum_{n \leq N} 2^n c_n(s) \leq 2^{N} c_N(s) \sum_{n\leq N} 2^{2(n-N)/3} \leq \frac{1}{2^{2/3}-1} 2^{N} c_N(s).\] If $N \geq 0$ is the smallest integer such that $2^{-N} \leq |s-t|$, 
  \[ |\varphi(s)-\varphi(t)| \lesssim  \sum_{n \geq N} c_n(s)+c_n(t),\]
  which by \eqref{eq:bound_cns} proves the corollary.
\end{proof}

\begin{cor}\label{cor:Spbddness_implies_integrable_modulus}
  If $d \in \N, 1 \leq p \leq \infty$ are such that the function
  \[ \delta \mapsto \frac{\fonction_d(\delta)^{|\frac 1 p - \frac 1 2|} \log |\log \delta|}{\delta}\]
  is integrable at $0$, then every function $m:(-1,1) \to \C$ such that the Schur multiplier with symbol $(\xi,\eta)\mapsto m(\langle \xi,\eta\rangle)$ is bounded on $S_p(L_2(\R^d))$ with norm $\leq 1$ coincides almost everywhere with a continuous function whose modulus of continuity $\omega$ at $0$ is independent of $m$ and such that $t\mapsto \frac{\omega(t)}{t|\log t|}$ is integrable at $0$.

  In particular, if $\Gamma$ is $\SL_{2d-1}(\Z)$ or any group as in Theorem~\ref{thm:rank0reductionImproved}, then $L_p(\LL \Gamma)$ does not have the Operator space Approximation Property.
\end{cor}
\begin{proof}
  The consequence in terms of OAP is Theorem~\ref{thm:rank0reductionImproved}. For the first part, by Corollary~\ref{cor:Lpbddness_LP_quantitative_precise} we have to show that
  \[ \omega(t):= \int_0^{t^{1/3}} \frac{\fonction_d(\delta)^{|\frac 1 p - \frac 1 2|}}{\delta} d\delta + t^{\frac 1 9} \]
  satisfies that $t\mapsto \frac{\omega(t)}{t|\log t|}$ is integrable at $0$. The $t^{\frac 1 9}$ term does not cause any problem, so we compute the integral
  \begin{align*} \int_0^{1/e^3} \frac{1}{t |\log t|} \int_0^{t^{1/3}} \frac{\fonction_d(\delta)^{|\frac 1 p - \frac 1 2|}}{\delta} d\delta    & = \int_0^{1/e} \frac{\fonction_d(\delta)^{|\frac 1 p - \frac 1 2|}}{\delta} \int_{\delta^3}^{1/e^3} \frac{1}{t|\log t|} dt d\delta\\
    & =\int_0^{1/e} \frac{\fonction_d(\delta)^{|\frac 1 p - \frac 1 2|}}{\delta} \log |\log \delta | d\delta<\infty.\qedhere
  \end{align*}
\end{proof}

\bibliographystyle{plain}
\bibliography{biblio.bib}

\begin{thebibliography}{10}

\bibitem{MR2628799}
A.~B. Aleksandrov and V.~V. Peller.
\newblock Operator {H}\"older-{Z}ygmund functions.
\newblock {\em Adv. Math.}, 224(3):910--966, 2010.

\bibitem{MR1357166}
Stefan Banach.
\newblock {\em Th\'eorie des op\'erations lin\'eaires}.
\newblock \'Editions Jacques Gabay, Sceaux, 1993.
\newblock Reprint of the 1932 original.

\bibitem{MR1354598}
H.~Brezis and L.~Nirenberg.
\newblock Degree theory and {BMO}. {I}. {C}ompact manifolds without boundaries.
\newblock {\em Selecta Math. (N.S.)}, 1(2):197--263, 1995.

\bibitem{CaspersdlS}
Martijn Caspers and Mikael de~la Salle.
\newblock Schur and {F}ourier multipliers of an amenable group acting on
  non-commutative {$L^p$}-spaces.
\newblock {\em Trans. Amer. Math. Soc.}, 367(10):6997--7013, 2015.

\bibitem{MR4660138}
Jos\'e{}~M. Conde-Alonso, Adri\'an~M. Gonz\'alez-P\'erez, Javier Parcet, and
  Eduardo Tablate.
\newblock Schur multipliers in {S}chatten--von {N}eumann classes.
\newblock {\em Ann. of Math. (2)}, 198(3):1229--1260, 2023.

\bibitem{MR272988}
Roy~O. Davies.
\newblock Some remarks on the {K}akeya problem.
\newblock {\em Proc. Cambridge Philos. Soc.}, 69:417--421, 1971.

\bibitem{dlS_habil}
Mikael de~la Salle.
\newblock Rigidity and malleability aspects of groups and their
  representations.
\newblock Habilitation thesis, ENS Lyon,
  \url{http://perso.ens-lyon.fr/mikael.de.la.salle/HDR_delaSalle.pdf}, 2016.

\bibitem{MR4680356}
Mikael de~la Salle.
\newblock Analysis on simple {L}ie groups and lattices.
\newblock In {\em I{CM}---{I}nternational {C}ongress of {M}athematicians.
  {V}ol. 4. {S}ections 5--8}, pages 3166--3188. EMS Press, Berlin, [2023]
  \copyright 2023.

\bibitem{MR3035056}
Tim de~Laat.
\newblock Approximation properties for noncommutative {$L^p$}-spaces associated
  with lattices in {L}ie groups.
\newblock {\em J. Funct. Anal.}, 264(10):2300--2322, 2013.

\bibitem{MR3781331}
Tim de~Laat and Mikael de~la Salle.
\newblock Approximation properties for noncommutative {$L^p$}-spaces of high
  rank lattices and nonembeddability of expanders.
\newblock {\em J. Reine Angew. Math.}, 737:49--69, 2018.

\bibitem{zbMATH07962858}
Guillaume Dumas.
\newblock Regularity of matrix coefficients of a compact symmetric pair of
  {Lie} groups.
\newblock {\em Trans. Am. Math. Soc.}, 377(10):7421--7474, 2024.

\bibitem{MR402468}
Per Enflo.
\newblock A counterexample to the approximation problem in {B}anach spaces.
\newblock {\em Acta Math.}, 130:309--317, 1973.

\bibitem{zbMATH03370942}
Charles Fefferman.
\newblock The multiplier problem for the ball.
\newblock {\em Ann. Math. (2)}, 94:330--336, 1971.

\bibitem{MR2463316}
Loukas Grafakos.
\newblock {\em Modern {F}ourier analysis}, volume 250 of {\em Graduate Texts in
  Mathematics}.
\newblock Springer, New York, second edition, 2009.

\bibitem{MR94682}
A.~Grothendieck.
\newblock R\'{e}sum\'{e} de la th\'{e}orie m\'{e}trique des produits tensoriels
  topologiques.
\newblock {\em Bol. Soc. Mat. S\~{a}o Paulo}, 8:1--79, 1953.

\bibitem{MR0075539}
Alexandre Grothendieck.
\newblock Produits tensoriels topologiques et espaces nucl\'{e}aires.
\newblock {\em Mem. Amer. Math. Soc.}, No. 16:Chapter 1: 196 pp.; Chapter 2:
  140, 1955.

\bibitem{MR3047470}
Uffe Haagerup and Tim de~Laat.
\newblock Simple {L}ie groups without the approximation property.
\newblock {\em Duke Math. J.}, 162(5):925--964, 2013.

\bibitem{MR2784663}
Yaryong Heo, F\"edor Nazarov, and Andreas Seeger.
\newblock Radial {F}ourier multipliers in high dimensions.
\newblock {\em Acta Math.}, 206(1):55--92, 2011.

\bibitem{zbMATH01344340}
U.~Keich.
\newblock On {{\(L^p\)}} bounds for {Kakeya} maximal functions and the
  {Minkowski} dimension in {{\(\mathbb{R}^2\)}}.
\newblock {\em Bull. Lond. Math. Soc.}, 31(2):213--221, 1999.

\bibitem{Lafforgue08}
Vincent Lafforgue.
\newblock Un renforcement de la propri\'{e}t\'{e} ({T}).
\newblock {\em Duke Math. J.}, 143(3):559--602, 2008.

\bibitem{LafforguedlS}
Vincent Lafforgue and Mikael de~la Salle.
\newblock Noncommutative {$L^p$}-spaces without the completely bounded
  approximation property.
\newblock {\em Duke Math. J.}, 160(1):71--116, 2011.

\bibitem{MR2866074}
Stefan Neuwirth and \'Eric Ricard.
\newblock Transfer of {F}ourier multipliers into {S}chur multipliers and
  sumsets in a discrete group.
\newblock {\em Canad. J. Math.}, 63(5):1161--1187, 2011.

\bibitem{parcet2025impactschurmultipliersharmonic}
Javier Parcet.
\newblock The impact of schur multipliers in harmonic analysis and operator
  algebras, proceedings of the 2026 icm, 2025.

\bibitem{MR4882285}
Javier Parcet, Mikael de~la Salle, and Eduardo Tablate.
\newblock The local geometry of idempotent {S}chur multipliers.
\newblock {\em Forum Math. Pi}, 13:Paper No. e14, 21, 2025.

\bibitem{ParcetRicarddlS}
Javier Parcet, \'{E}ric Ricard, and Mikael de~la Salle.
\newblock Fourier multipliers in {SL}n({R}).
\newblock {\em Duke Math. J.}, 171(6):1235--1297, 2022.

\bibitem{zbMATH03158261}
Jaak Peetre.
\newblock Rectification {\`a} l'article ''{Une} caract{\'e}risation abstraite
  des op{\'e}rateurs diff{\'e}rentiels''.
\newblock {\em Math. Scand.}, 8:116--120, 1960.

\bibitem{MR1648908}
Gilles Pisier.
\newblock Non-commutative vector valued {$L_p$}-spaces and completely
  {$p$}-summing maps.
\newblock {\em Ast\'erisque}, (247):vi+131, 1998.

\bibitem{PisierOS}
Gilles Pisier.
\newblock {\em Introduction to operator space theory}, volume 294 of {\em
  London Mathematical Society Lecture Note Series}.
\newblock Cambridge University Press, Cambridge, 2003.

\bibitem{PisierXu}
Gilles Pisier and Quanhua Xu.
\newblock Non-commutative {$L^p$}-spaces.
\newblock In {\em Handbook of the geometry of {B}anach spaces, {V}ol. 2}, pages
  1459--1517. North-Holland, Amsterdam, 2003.

\bibitem{MR631090}
Andrzej Szankowski.
\newblock {$B({H})$} does not have the approximation property.
\newblock {\em Acta Math.}, 147(1-2):89--108, 1981.

\bibitem{Vergara}
I.~Vergara.
\newblock The {$p$}-approximation property for simple {L}ie groups with finite
  center.
\newblock {\em J. Funct. Anal.}, 273(11):3463--3503, 2017.

\bibitem{Wang_2025}
Hong Wang and Joshua Zahl.
\newblock Sticky kakeya sets and the sticky kakeya conjecture.
\newblock {\em Journal of the American Mathematical Society}, November 2025.

\bibitem{wang2025volumeestimatesunionsconvex}
Hong Wang and Joshua Zahl.
\newblock Volume estimates for unions of convex sets, and the kakeya set
  conjecture in three dimensions, 2025.

\bibitem{zbMATH03096429}
Antoni Zygmund.
\newblock Smooth functions.
\newblock {\em Duke Math. J.}, 12:47--76, 1945.

\end{thebibliography}

\end{document}